\documentclass[oneside,english]{article}
\usepackage[T1]{fontenc}
\usepackage{amsthm}
\usepackage{amssymb}
\usepackage[all]{xy}
\usepackage{amsmath}
\usepackage{comment}
\usepackage{color}
\usepackage[nocompress]{cite}
\usepackage{tikz-cd}
\usepackage{mathtools}
\usepackage{covington}

\makeatletter
\newcommand{\nc}{\newcommand}
\newtheorem{thm}{Theorem}
\theoremstyle{plain}
\nc{\bthm}{\begin{thm}} \nc{\ethm}{\end{thm}}
\newtheorem{prop}[thm]{Proposition}
\nc{\bprp}{\begin{prop}} \nc{\eprp}{\end{prop}}
\newtheorem{fact}[thm]{Fact}
\nc{\bfct}{\begin{fact}} \nc{\efct}{\end{fact}}
\newtheorem{prob}[thm]{Problem}
\nc{\bprb}{\begin{prob}} \nc{\eprb}{\end{prob}}
\newtheorem{lem}[thm]{Lemma}
\nc{\blem}{\begin{lem}} \nc{\elem}{\end{lem}}
\newtheorem{claim}[thm]{Claim}
\nc{\bclm}{\begin{claim}} \nc{\eclm}{\end{claim}}
\newtheorem{cor}[thm]{Corollary}
\nc{\bcor}{\begin{cor}} \nc{\ecor}{\end{cor}}
\newtheorem{conj}[thm]{Conjecture}
\nc{\bcnj}{\begin{conj}} \nc{\ecnj}{\end{conj}}
\theoremstyle{definition}
\newtheorem{defn}[thm]{Definition}
\nc{\bdfn}{\begin{defn}} \nc{\edfn}{\end{defn}}
\newtheorem{observation}[thm]{Observation}
\nc{\bobs}{\begin{observation}} \nc{\eobs}{\end{observation}}
\theoremstyle{remark}
\newtheorem{rem}[thm]{Remark}
\nc{\brem}{\begin{rem}} \nc{\erem}{\end{rem}}
\newtheorem{cnv}[thm]{Convention}
\nc{\bcnv}{\begin{cnv}} \nc{\ecnv}{\end{cnv}}
\newtheorem{exam}[thm]{Example}
\nc{\bexm}{\begin{exam}} \nc{\eexm}{\end{exam}}

\newtheorem{question}[thm]{Questuion}
\nc{\bpf}{\begin{proof}} \nc{\epf}{\end{proof}}
\nc{\be}{\begin{enumerate}}
	\nc{\ee}{\end{enumerate}}
\nc{\bi}{\begin{itemize}}
	\nc{\itm}{\item}
	\nc{\ei}{\end{itemize}}
\nc{\invlim}{\lim_{\leftarrow}}
\nc{\dirlim}{\lim_{\rightarrow}}
\nc{\mm}{\mathbf{m}}
\nc{\nn}{\mathbf{n}}
\nc{\FF}{\mathcal{F}}
\nc{\CC}{\mathcal{C}}
\nc{\Span}{\operatorname{span}}
\nc{\Img}{\operatorname{Im}}
\nc{\rank}{\operatorname{rank}}
\nc{\proj}{\operatorname{proj}}
\nc{\F}{\mathbb{F}}
\nc{\Z}{\mathbb{Z}}
\nc{\Q}{\mathbb{Q}}
\nc{\Br}{\operatorname{Br}}

\title{ Demushkin groups of uncountable rank}

\author{Tamar Bar-On and Nikolay Nikolov}
\date{\today}
\begin{document}
	\maketitle
	\begin{abstract}
		We extend the theory of countably generated Demushkin groups to Demushkin groups of arbitrary rank. We investigate their algebraic properties and invariants, count their isomorphism classes and study their realization as absolute Galois group. At the end, we compute their profinite completion and conclude with some results on profinite completion of absolute Galois groups.
	\end{abstract}
	\section*{Introduction}
	Let $G$ be a finitely generated pro-$p$ group. We say that $G$ is a \textit{Demushkin} group if the following two properties are satisfied: \begin{enumerate}
		
		\item $\dim H^2(G)=1$
		\item The cup product yields a nondegenerate bilinear form $H^1(G)\times H^1(G)\to H^2(G)$. 
	\end{enumerate}
	Here and below we shall write $H^n(G)$ for $H^n(G,\F_p)$.
	
	Infinite Demushkin groups are precisely the Poincare-Duality pro-$p$ groups of dimension 2. In particular they have cohomological dimension $2$, and satisfy that every open subgroup is itself Demushkin. Demushkin groups play an important role in number theory. Serre proved in \cite{serre1962structure}, \cite{serre1979galois}, that the maximal pro-$p$ 	Galois groups of finite extensions of $\mathbb{Q}_p$ which contain primitive root of identity are all Demushkin groups.It is still an open question whether every finitely generated Demushkin group can be realized as a maximal pro-$p$ Galois group of a field. Here by maximal pro-$p$ Galois group of a field we mean the maximal pro-$p$ quotient of its absolute Galois group.
	
	Since $\dim(H^2(G))=1$, $G$ can be presented as a quotient $F/\langle r^F \rangle$ of a free pro-$p$ group of rank $\dim(H^1(G))$, by the normal subgroup generated by one relator $r\in \Phi(F)$.  The structure of $r$ was completely classified in \cite{demushkin1961group},\cite{demushkin19632} and \cite{labute1967classification}. Moreover, Demushkin groups are completely determined by two invariants defined as follows: Let $q(G)=p^h$ for $h$ the maximal natural integer such that $r\in F^{p^h}[F,F]$ or $q(G)=0$ in case $r\in [F,F]$. Notice that this definition doesn't depend on the choice of $r$, as $G/[G,G]\cong \mathbb{Z}_p/q(G)\times (\mathbb{Z}_q)^{\rank(G)-1}$. If $q\ne 2$ then for every natural number $n> 1$ there is a unique Demushkin group of rank $n$ for which $q(G)=q$. If $q=2$ the situation is more complicated.  A second invariant was defined by Serre for this case: Being a Poincare duality group, every infinite Demushkin group comes equipped with a dualizing module $I$, which is isomorphic to $\mathbb{Q}_p/\mathbb{Z}_p$. Hence, we denote by $\chi:G\to I^{\times}$ the induced homomorphism, and define the invariant $\Img(\chi)$. 
	
	In \cite{labute1966demuvskin} Labute generalized the theory of Demushkin groups to Demushkin groups of rank $\aleph_0$. By a Demushkin group of rank $\aleph_0$ he referred to  pro-$p$ groups $G$ satisfying the conditions 1 and 2 above, and such that $\dim(H^1(G))=\aleph_0$. He gave an almost full characterization, except to one case, for which he defined two more invariants: $s(G)$ and $t(G)$. Since Poincare-Duality groups are necessarily finitely generated, a Demushkin group of rank $\aleph_0$ is no longer a Poincare duality group. However, it can be shown that it has a dualizing module $I$, which is isomorphic to either $\mathbb{Q}_p/\mathbb{Z}_p$ or to $\mathbb Z/ q\mathbb Z$ for an integer $q$ which is a prime power of $p$. We define the invariant $s(G)$ to be zero in the former case and equal to $q$ in the latter case. Labute showed that for every $q\ne 2$ and $s\geq q$ or zero, there exists a unique Demushkin group of countable rank $G$ such that $q(G)=q$ and $s(G)=s$. The invariant $t(G)$, which is only interesting in the case $q\ne 2$ will be discussed later. In his paper \cite{labute1966demuvskin} Labute also proved that the $p$-Sylow subgroups of the absolute Galois groups of local fields containing  a primitive root of unity are Demushkin groups of countable rank. This work was completed by Minac and Ware, who proved in \cite{minavc1991demuvskin} and \cite{minavc1992pro} that for every prime $p\ne 2$ a pro-$p$ Demushkin group of rank $\aleph_0$ occurs as an absolute Galois group of some field if and only if $s(G)=0$. Moreover, a pro-$2$ Demushkin group $G$ of rank $\aleph_0$ occurs as an absolute Galois group if and only if $s(G)=0$ and the couple  $(t(G),\Img(\chi))$ belongs to a specific list of possibilities.
	
	The object of this paper is to present a theory of Demushkin group of uncountable rank, and discuss their applications to number theory. Let $\mu>\aleph_0$ be some cardinal. We say that a pro-$p$ group with $\dim(H^1(G))=\mu$ is Demushkin, if it satisfies the conditions 1 and 2 at the start of the paper. 
	
	The paper is organized as follows: In Section 1 we give some necessary background from Galois cohomology and bilnear form theory. In Section 2 we discuss the algebraic properties of a Demushkin group of uncountable rank, and in particular prove the following:
	\begin{thm}
		Let $q\ne 2$ be a prime power or equal to 0, and $\mu>\aleph_0$. For every nondegenerate skew symmetric bilinear form $(V,\varphi)$ of dimension $V$ there is a Demushkin group $G$ of rank $\mu$, with $q(G)=q, s(G)=0$, and whose cup-product bilinear form $H^1(G)\cup H^1(G)\to H^2(G)$ is isomorphic to $(V,\varphi)$.
	\end{thm}   
	\begin{cor}
		For every cardinal $\mu>\aleph_0$, $p$ a prime, and $q\ne 2$ a power of $p$ or zero, there are $2^{\mu}$ pairwise nonisomorphic Demushkin pro-$p$ groups $G$ of rank $\mu$ with $q(G)=q, s(G)=0$.
	\end{cor}
	In addition, for the $q=2$ case, we have the following:
	\begin{thm}
		For every choice of $t\in\{-1,0,1\}$, a subgroup $A$ of $\mathbb{Z}_2^{\times}$ and uncountable cardinal $\mu$, there are $2^{\mu}$ pairwise nonisomorphic pro-$2$ Demushkin groups $G$ of rank $\mu$ with $s(G)=0, t(G)=t$ and $\Img(\chi)=A$. 
	\end{thm}
	We also give the following two characterizations of Demushkin groups of arbitrary rank:
	\begin{thm}
		Let $G$ be a pro-$p$ group. The following are equivalent:
		\begin{enumerate}
			\item $G$ is a Demushkin group.
			\item $\operatorname{cd}(G)=2$ and the dualizing module $I$ of $G$ has a unique additive subgroup of size $p$.
		\end{enumerate}
	\end{thm}
	\begin{thm}
		Let $G$ be a pro-$p$ one relator group of arbitrary rank. Then $G$ is a Demushkin group if and only if, every open subgroup $U\leq G$ of index $p$ is 1-related.
	\end{thm} 
	In Section 2 we discuss the realization of Demushkin groups of uncountable rank as maximal pro-$p$ quotients of absolute Galois groups. This work is motivated by the wide version of Inverse Galois Problem, which asks to determine which profinite groups can be realized as absolute Galois groups, and its restricted pro-$p$ version, which asks to determine which pro-$p$ groups can be realized as maximal pro-$p$ quotients of absolute Galois groups. In particular, we prove the following:
	\begin{thm}
		For every cardinal $\mu>\aleph_0$, prime $p$, and $q$ a power of $p$ or $0$, there exists a field $F$ whose absolute Galois group is a Demushkin group $G$ of rank $\mu$, $q(G)=q$ and $s(G)=0$.
	\end{thm}
	For the case $p=2$ case we need notation for the subgroups of $\mathbb Z_2^ \times$. For an integer $f \geq 2$ define $U_2^{(f)}$, respectively $U_2^{[f]}$, to be the closed subgroups generated by $1+2^f$, respectively $-1+2^f$, in $\mathbb Z_2^\times$. We have the following result.
	
	\begin{thm}
		Let $G$ be a pro-2 Demushkin group. If $G$ is an absolute Galois group, then $t(G)\ne 0$. In addition, for every cardinal $\mu$, there exist an absolute Galois group over a field of characteristic 0 which is a pro-2 Demushkin group $G$ of rank $\mu$ for every pair of invariants in the following list, and only for such invariants.
		\begin{itemize}
			\item $t(G)=1$ and  $\operatorname{Im}(\chi) \in \{ U_2^{(f)},U_2^{[f]}, f \in \mathbb N, f \geq 2\}$,
			
			\item $t(G) \in \{-1,1\}$,$ \operatorname{Im}(\chi)=\{\pm 1\} U_2^{(f)}, f \in \mathbb N, f \geq 2$.
		\end{itemize}
		
		If $F$ is a field of characteristic $p$, whose absolute Galois group is a pro-2 Demushkin group $G$, then $t(G)=1$ and exactly the following options of $\operatorname{Im}(\chi)$ are possible:
		\begin{itemize}
			\item  If $p\equiv 1 (mod4)$, say $p = 1 + 2^ac, a > 2, 2\nmid c$, then $\operatorname{Im}(\chi)=U_2^{(b)}$ for some  $a \leq b < \infty$.	 Moreover, each group
			$U_2^{(b)}, a \leq b < \infty$, occurs as $\operatorname{Im}(\chi)$ for some Demushkin group of rank $\mu$ $G_F$ with $\operatorname{char}(F)=p$. 
			\item If $p\equiv -1 (mod4)$, say $p = -1 + 2^ac, a > 2, 2\nmid c$, then the only possibilities for $\operatorname{Im}(\chi)$
			are the groups $U_2^{[a]}$, and  $U_2^{(b)}$, with $a+1 \leq b < \infty$.	 Moreover, each of these groups can occur as $\operatorname{Im}(\chi)$ for some Demushkin group $G_F$ of rank $\mu$ with $\operatorname{char}(F)=p$.
		\end{itemize} 
	\end{thm}
	In Section 4 we compute the profinite completion of a Demushkin group of infinite rank, it turns out to be always a free pro-$p$ group. Bearing in mind that free pro-$p$ groups of any rank are absolute Galois groups, we get a new class of examples of absolute Galois groups such that their profinite completion remain absolute Galois group, as been discussed in \cite{baron-cohomological}.
	
	\section*{Theoretical background}
	Recall that $\dim(H^1(G))$ equals the cardinality of a minimal set of generators of $G$. In particular, $H^1(G)=(G/\Phi(G))^*$, and for every subset $X=\{x_i\}\subseteq G$, $X$ is a minimal set of generators for $G$ if and only if $\{x_i\mod \Phi(G)\}$ is a minimal set of generators for $G/\Phi(G)$. Hence, for every minimal set of generators $\{x_i\}$ of $G$, we can define a dual basis $\{\chi_i\}$ of $H^1(G)$ by $\chi_i(x_j)=\delta_{ij}$. Conversely, for every basis of $\{\chi_i\}$ of $H^1(G)$ one can construct a dual minimal set of generators $\{x_i\}$ for $G$, satisfying $\chi_i(x_j)=\delta_{ij}$. 
	
	As $\dim(H^2(G))$ equals the cardinality of a minimal set of relations defining $G$, $H^2(G)\cong \F_p$ if and only if $G$ is 1-relator, i.e, $G\cong F/\langle \langle r \rangle \rangle$ is a quotient of a free pro-$p$ group with the same rank, by a normal subgroup generated by one element $r\in \Phi(G)$. For 1-relator pro-$p$ groups, the cup product bilinear form $H^1(G)\times H^1(G)\to H^2(G)\cong \F_p$ has the following interpretation: choose a basis $\{x_i\}$ for $F$. Since $r\in \Phi(G)$, $r\equiv r'\mod [[F,F],F]$ where $r'=\prod x_i^{p\alpha_i}\prod_{i<j}[x_i,x_j]^{\beta_{i,j}}$. Take $\{\chi_i\}$ to be the dual basis, then $\chi_i\cup \chi_j=\beta_{ij}\mod p$ for every $i<j$ and $\chi_\cup \chi=\binom{p}{2}\mod p$. The cup product bilinear form is skew-symmetric. 
	
	Now we give a list of results regarding nondegenerate bilinear forms. Our motivation is the cup product bilinear form and hence we will only refer to skew symmetric bilinear forms. Recall that a bilinear form $\varphi:V\times V\to \F$ on a vector space $V$ is nondegenerate if for every $v\in V$ there exits $u\in V$ such that $\varphi(v,u)\ne 0$. In addition, over a field with character different then 2, every skew symmetric bilinear form is \textit{alternate}, meaning that for every $v\in V$, $\varphi(v,v)=0$. It is a well known fact that  every finite dimensional vector space with a nondegenerate alternate bilinear form has a \textit{symplectic basis}, meaning a basis of the form $\{x_1,...,x_{2n}\}$ which satisfies $\varphi(x_{2i-1},x_{2_i})=1$ and for all other pairs $i\leq i, \varphi(x_i,x_j)=0$. Hence up to isometry there is only one alternate bilinear form on a even dimensional vector space $V$ over a field of characteristic different than 2. The characteristic 2 case is a bit more complicated, but still on each finite dimensional vector space $V$ there are only finitely many pairwise nonisometric nondegenerate skew symmetric forms. As for the infinite dimension case, we have the following:
	\begin{lem}\label{Projections}
		Let $V$ be a vector space a some field $F$, equipped with a skew-symmetric bilinear form $\varphi:V\times V\to F$, and let $U=\Span \{v_1,...,v_n\}$ be a finite dimensional subspace. Then $V=U\oplus U^{\perp}$ where $U^{\perp}$ denotes the orthogonal complement of $U$.
	\end{lem}
	\begin{proof}
		It is enough to prove that for every $v\in V$ there are some scalars $\alpha_1,...,\alpha_n\in F$ such that $\varphi(v-\sum \alpha_iv_i,v_j)=0$ for every $j=1,...,n$. This is equivalent to solve the equation system $\sum \alpha_i\varphi(v_i,v_j)=\varphi(v,v_j)$ which follows immediately by the nondegeneracy of $\varphi$ on $U$, which is equivalent to the matrix $(\varphi(v_i,v_j))$ being invertible. Later on we will denote the element in $u\in U$ for which $v-u\in U^{\perp}$ by the projection of $v$ on $U$.   
	\end{proof}
	As a result, we get the following useful proposition:
	\begin{prop}\label{locally nondegenerate}
		A skew-symmetric bilinear form over an infinite dimensional vector space is nondegenerate if and only if it is locally nondegenerate.
	\end{prop}
	\begin{proof}
		The second direction is trivial, so we only prove the first direction.
		
		Let $V$ be an infinite dimensional vector space over some field $F$, equipped with a nondegenerate bilinear form $\varphi:V\times V\to F$. We want to prove that every finite subset $A\subseteq V$ is contained in some finite dimensional nondegenerate subspace $U$. We prove it by induction on the size of $A$. When $A=\{v\}$, then either $\varphi(v,v)\not =0$ and we can take
		$U=\operatorname{span}(v)$, or there exist some $u\in V$ such that $\varphi(u,v)\not =0$, and then we can take $U=\Span \{v,u\}$. Now let $A=\{v_1,...,v_n\}$.  If $\varphi(v_1,v_1)\ne 0$, denote $V'= \Span \{v_1\}$. Otherwise, we might add an element $v_{n+1}$ such that, after renamimg, $\varphi(v_1,v_2)\ne 0$. In that case, denote  $V'=\Span\{v_1,v_2\}$. Notice that $V'$ is nondegenerate, and hence $V=V'\oplus V'^{\perp}$.  $v_i-{\proj}_{V'}(v_j)\in V'^{\perp}$ for all $v_i\in A\setminus V'$. Since $V$ is nondegenerate, so is $V'^{\perp}$. Hence we can apply the induction hypothesis on $\{v_i-{\proj}_{V'}(v_j)\}_{v_i\in A\setminus V'}$ inside $V'^{\perp}$ to get some  nondegenerate finite dimensional subspace $W\subseteq V'^{\perp}$ which contains $\{v_i-{\proj}_{V'}(v_j)\}_{v_i\in A\setminus V'}$ inside $V'^{\perp}$. Taking $U=V'\oplus W$, we are done.
	\end{proof}
	Lemma \ref{Projections} implies the next result, which was first stated in \cite{kaplansky1950forms}:
	\begin{lem}
		Let $\varphi$ be an alternate nondegenerate bilinear form over a vector space $V$ of countable dimsension. Then $V$ has a symplectic basis. Henceforth, there is only one alternate nondegenerate bilinear form of infinite countable dimension over a given field.
	\end{lem}
	Unlike the countable case,  when $\dim V$ is uncountable, we have the following Theorem:
	\begin{thm}\cite[Main Theorem]{hall1988number}\label{maximal number of forms}
		Let $\mu>\aleph_0$ be a cardinal. For every field $\F$ there exist $2^{\mu}$ pairwise nonisometric nondegenerate alternate bilinear forms of dimension $\mu$ over $\F$.  
	\end{thm}
	As skew symmetric bilinear forms of characteristic 2 may not be alternate, we need to handle this case separately. There are 3 kinds of skew-symmetric bilinear forms which are characterized by the invariant $t(\varphi)$ as follows:
	\begin{defn}
		Denote $\beta:V\to \F$ by $\beta(v)=\varphi(v,v)$. Let $A=\ker (\beta)$, and $A^{\perp}$ be the orthogonal complement.
		\begin{itemize}
			\item If $A=V$ define $t(\varphi)=1$.
			\item If $A\ne V$ and $0\ne A^{\perp}\subseteq A$ define $t(\varphi)=1$.
			\item If $A\ne V$ and $A^{\perp}\nsubseteq A$ define $t(\varphi)=-1$.
			\item If $A^{\perp}=0$ define $t(\varphi)=0$.
		\end{itemize}
	\end{defn}    
	\begin{thm}\label{forms over F_2}
		Let $\mu>\aleph_0$ be a cardinal. For every $t\in \{-1,0,1\}$, there are $2^{\mu}$ pairwise nonisometric nonalternate nondegenerate skew-summetric bilinear forms $\varphi$ over $\F_2$ of dimension $\mu$, such that $t(\varphi)=t$.
	\end{thm}
	\begin{proof}
		This proof is based on the proof of \cite[Main Theorem]{hall1988number}. Let $t\in \{-1,0,1\}$. Consider the first order language $L$ with 
		\begin{enumerate}
			\item Two unary relation symbols $\FF$ and $\mathcal{V}$;
			\item Five ternary relation symbols $AF, MF, AV, SM, F$;
			\item A constant 0;
			\item A constant 1, for $1\in \F_2$.
		\end{enumerate}
		From $L$ we fashion a theory of nondegenerate skew-symmetric bilinear forms $\varphi$ over $\F_2$ with $t(\varphi)=t$. For a model $M$ of this theory, $M$
		will be composed of a copy of $\F_2$ given by $\FF_M$, an $\F_2$-vector space given by $\mathcal{V}_M$, and a bilinear form from $\mathcal{V}_M\times \mathcal{V}_M$ to $\mathcal{F}_M$ described by $F_M$. More precisely, we axiomatize  the theory using first order sentences of $L$. We require $\FF_M$ and $\mathcal{V}_M$ to be disjoint subsets of the model $M$ whose union is all
		of $M$, where $\FF_m$ is isomorphic to $\F_2$ and $\mathcal{V}_M$ is a vector space over $\F_M$. Here addition
		and multiplication in $\FF$ are described by the relations: $(\alpha,\beta,\alpha+\beta)\in AF$ and $(\alpha,\beta,\alpha\beta)\in MF$. Vector addition and scalar multiplication are described by $(v,u,v+u)\in AV$ and $(\alpha,v,\alpha v)\in SM$. Now $F$ describes a bilinear form $\varphi$ on $\mathcal{V}$ via $(v,u,\varphi(v,u))\in F$. We can require $\varphi$ to be skew-symmetric, nondegenerate, nonalternate with $t(\varphi)=t$. Eventually, we require $\FF$ to be isomorphic to $\F_2$ by the first order sentence: $\forall x\in \FF (x=0\lor x=1)$. We got the axioms of a theory $T$ of nondegenerate nonalternate skew-symmetric bilinear forms $\varphi$ over $\F_2$ with $t(\varphi)=t$.
		
		To every model $M$ of $T$ there is associated bilinear form as in the theorem. Conversely, every bilinear form as in the
		theorem gives rise to a model of $T$. It is not difficult to check that two
		models of $T$ are isomorphic if and only if the corresponding bilinear forms are
		isometric. Now we want to apply Shelah's results on instability \cite{shelah1972combinatorial} to $T$ and finish the proof. In order to do that, we need to show that our theory $T$ is unstable. By \cite[Theorem 2.4.2]{shelah1972combinatorial}, it is enough to show that for every infinite cardinal $\mu$, there is a model $M$ of $T$ which embeds $(\mu,<)$ via a formula of $T$. We will build such a model $M$ for every $t\in \{-1,0,1\}$ as follows: 	Let $\mu$ be a cardinal. Let $V$ be the vector space over $\mathbb{F}_2$ with basis $\{u_i,v_i:i\in \mu\}$. We define a skew-symmetric bilinear form on $V$ by the following rules:
		\begin{enumerate}
			\item $\varphi(v_i,v_j)=0, \forall i\ne j$. $\varphi(v_i,v_i)=1 \forall i$ $\varphi(u_i,u_j)=0, \forall i\ne j$. $\varphi(u_i,u_i)=1 \forall i$. $\varphi(v_i,u_j)=1, \forall i< j$. $\varphi(v_i,u_j)=0, \forall i\geq  j$.
			
			This is a nondegenerate form. Indeed, let $v=\sum_{i\in J_1} v_i+\sum _{i\in J_2}u_i$ be an arbitrary element in $V$. If $j_1=\max J_1\geq \max J_2=j_2$, then $\varphi(v,v_{j_1})=\varphi(v_{j_1},v_{j_1})=1$. Otherwise, $\varphi(u_{j_2},v)=1+|J_1|\mod 2$. Take $i>j_2$ $\varphi(u_i,v)=|J_1|\mod 2$. One of them must be nonzero. Now $\ker(\beta)$ equals to all the sums of even number of basis elements, and one checks immediately that its orthogonal complement is trivial. Hence we got a nondegenerate bilinear form with $t=0$.
			\item 
			
			$\varphi(v_i,v_j)=0, \forall i,j$. $\varphi(u_i,u_j)=0, \forall (i,j)\ne (0,0)$. $\varphi(u_0,u_0)=1$, $\varphi(v_i,u_j)=1, \forall i<j$. $\varphi(v_i,u_j)=0, \forall i\geq j$.

			\item $\varphi(v_i,v_j)=0, \forall i,j$. $\varphi(u_i,u_j)=0, \forall (i,j)\ne (0,0)$. $\varphi(u_0,u_0)=1$, $\varphi(v_i,u_j)=1, \forall i>j$. $\varphi(v_i,u_j)=0, \forall i\leq j$. One easily checks that this is a nondegenerate bilinear form with $t=-1$.

		\end{enumerate}
		Notice that $(\mu,<)$ can be embedded into $M$ via the formula $i<j\iff M\models \lnot \delta[(u_i,v_i),(u_j,v_j)]$ in the first case, and $i<j\iff M\models  \delta[(u_i,v_i),(u_j,v_j)]$ in the second and third cases, where $\delta[(x,y),(z,w)]\iff (y,z,0)\in F$.
		
		Now by \cite[Theorem 2.6]{shelah1972combinatorial}, for every $\mu>\aleph_0$ there are $2^{\mu}$ pairwise nonisomorphic models of $T$ of cardinality $\mu$. 
	\end{proof}
	\section*{Algebraic properties of Demushkin groups of uncountable rank}
	We start with arbitrary projective limits of Demushkin groups.
	\begin{prop}
		A projective limit of Demushkin groups of arbitrary rank is either Demushkin or free.
	\end{prop}
	\begin{proof}
		Let $G={\invlim}_i\{G_i,f_i\}$ be a projective limit of Demushkin groups. By definition, for every $i$, $H^2(G_i)\cong \F_p$. Hence $H^2(G)=0\lor \F_p$. If $H^2(G)=0$ then $\operatorname{cd}(G)\leq 1$ and we get that $G$ is free. So assume $H^2(G)=\F_p$. We can assume that the inflation maps $\operatorname{Inf}:H^2(G_i)\to H^2(G_j)$ are isomorphism for all $i\leq j$, and hence the inflation maps  $\operatorname{Inf}:H^2(G_i)\to H^2(G)$ are isomorphisms for all $i$. Since $H^1(G)={\dirlim}_iH^1(G_i)$ and the cup product commutes with direct limits, we get that $(H^1(G),\cup)$ is nondegenerate. Henceforth $G$ is a Demushkin group. Now we want to show that both of these outcomes indeed occur. Obviously, every Demushkin group can be expressed as the inverse limit of copies of itself with isomorphism. So we will construct an example of a projective limit of Demushkin groups with is free.
		
		Let $G=\lim G_n$ be a Demushkin group of rank $\aleph_0$ expressed as a strictly increasing projective limit of f.g Demushkin groups. This can be done due to \cite[Theorem 1]{labute1966demuvskin}. We claim that for every open subgroup $H_n\leq G_n$ there exist an open subgroup $H_{n+1}\leq G_{n+1}$ such that $\varphi(H_{n+1})=H_n$, where $\varpi$ is the given epimorphism $G_{n+1}\to G_n$, and $[G_{n+1}:H_{n+1}]>[G_n:H_n]$. For that it is enough to show that $H_{n+1}\ne \varphi^{-1}(H_n)$. By \cite[Lemma 2.8.15]{ribes2000profinite} there is a minimal subgroup $K\leq G_n$ such that $\varphi(K)=H_n$, and that $K$ satisfies $\ker(\varphi|_{K})\leq \Phi(K)$. So it is enough to show that $M=\varphi^{-1}(H_n)$ doesn't satisfy $\ker(\varphi|_{M})\leq \Phi(M)$ to get that it is not minimal. Assume $\ker(\varphi|_{M})\leq \Phi(M)$. Then $d(M)=d(H_n)$. However, for Demushkin groups we have formula for the rank of an open subgroup. $d(M)=[G_{n+1}:M](d(G_{n+1})-2)+2>[G_n:H_n](d(G_n)-2)+2=d(H_n)$ since $[G_{n+1}:M]=[G_n:H_n]$ and $d(G_{n+1})>d(G_n)$. So we can take such a series of open subgroups projecting on each other of increasing index, and get that the inverse limit is a closed subgroup of infinite index in a countably generated Demushkin group, and hence by \cite[Theorem 2]{labute1966demuvskin} it is free. 
	\end{proof}
	Now we present the most useful tool in the study of Demushkin groups of uncountable rank.
	\begin{thm}\label{inverse limit}
		Every Demushkin group is an inverse limit of finitely generated Demushkin groups.
	\end{thm}
	\begin{proof}
		Let $G=F /  \langle  r^F \rangle$ be a Demushkin group of infinite rank. 
		Let $X$ be a free basis of $F$ and let $X^*$ be the dual basis to $X$ in $H^1(G)$. 
		By definition, the vector space $H^1(G)$ has a  nondegenerate cup product bilinear form. 
		For a finite subset $J \subset X$ Lemma \ref{locally nondegenerate} gives a finite dimensional nonegenerate subspace $L=L_J $ of $H^1(G)$ which contains $J^*=\{x^*\}_{x\in J}\subseteq X^*$. Choose a basis $B^*_1$ for $L$ and enlarge it to a basis $B^*_1 \cup B^*_2$ of $H^1(G)$. Let $\bar B_1 \cup \bar B_2$ be the dual basis to $B^*_1 \cup B^*_2$ in $G/\Phi(G)$.
		Observe that 
		
		\[ \bar B_2 \subset \bigcap_{f \in L} \ker f  \subset \bigcap_{f \in J^*} \ker f,\]
		while $\cap_{f \in J^*} \ker f$ is the closed subgroup spanned by the images of $X-J$ in $G/\Phi(G)$.
		Now lift the basis $\bar B_1 \cup \bar B_2$ of $G/\Phi(G)$ to a basis $B_1 \cup B_2$ of $F$ such that $B_2$ belongs to the closed subgroup of $F$ generated by $X-J$.
		
		Let $F'$ be the free pro-$p$ group generated by $B_1$, and define an epimorphism $\varphi:F\to F'$ by letting $x \mapsto x$ for $x \in B_1$ and $x\mapsto 1$ for $x \in B_2$. Let $r'=\varphi(r)$. We get an epimorphism $G\to G_J:=F'/\langle r'^{F'} \rangle$, which we denote by $\varphi_J$. We claim that $G_J$ is a Demushkin group. Indeed, as $G_J$ is constructed from a free profinite group by one relation, $\dim H^2(G_J)\leq 1$.
		
		Since $\varphi:G\to G_J$ is surjective, $\operatorname{Inf}:H^1(G_J)\to H^1(G)$ is injective. Hence it maps $H^1(G_J)$ isomorphically onto $Span\{B^*_1\}=L$ in $H^1(G)$. The cup product on $H^1(G_L)$ is induced from the cup product on $H^1(G)$ via the inflation map. It follows that the cup product $H^1(G_J)\cup H^1(G_J)\to H^2(G_J)$ is nondegenerate, and $\operatorname{Inf}:H^2(G_J)\to H^2(G)$ is not the zero map, which implies that $\dim H^2(G_J)=1$. Hence $G_J$ is a a Demushkin group as claimed. 
		
		In order to show that $G=\invlim \{G_J,\varphi_J\}$ of all the groups of this form, we need to show that the intersection of all the kernels $\bigcap \{\ker \varphi_J\}$ is trivial. Let $H_{J^c}$ be closed normal subgroup of $G$ generated by the complement $J^c=X-J$ of $J$ in $X$. By construction $\ker \varphi_j$ is contained in $H_{J^c}$.
		It is a basic property of profinite group that $\bigcap H_{J^c}$ is trivial where $J$ runs over all finite subsets of $X$. The proof is complete.
	\end{proof}
	\begin{cor}\label{open subgroups}
		Let $G$ be a Demushkin group of arbitrary rank. Then $\operatorname{cd}(G)=2$, and every open subgroup of a Demushkin group is again Demushkin. Moreover, every closed subgroup of infinite index of a Demushkin is free. 
	\end{cor}
	\begin{proof}
		Using Theorem \ref{inverse limit}, this is identical to the proofs of \cite{labute1966demuvskin} Theorem 2, and the corollary to Theorem 1.	
	\end{proof}
	Now we want to study the invariants $q(G),s(G),\Img(\chi)$ of a Demushkin group of arbitrary rank. $q(G)$ will be defined as in the finite case, to be the maximal $q$ such that $r\in F^q[F,F]$ or $0$ if $r\in [F,F]$, when we express $G$ as $F/r$ for a free pro-$p$ group $F$. Observe that as $\operatorname{cd}(G)=2$, $G$ has a dualizing module $I$. As explained in \cite[page 4]{labute1966demuvskin}, $I=\Q_p/\Z_p$ or $I=\Z/qZ$ for some $p$-power $q$. In the first case we set $s(G)=0$, while in the second case we set $s(G)=q$. The dualizing module comes equipped with a character $\chi:G\to \operatorname{Aut}(I)\cong (Z_p/s(G))^{\times}$  hence we have the invariant $\Img(\chi)$. 	We shall need the properties $P$ and $Q$ introduced in \cite[Section 5]{labute1966demuvskin} as follows. 
	\begin{defn}
		Let $M$ be an abelian group with automorphism group isomorphic to $\Z_p/q\Z_p$ for $q$ a power of $p$ or 0, and let $\chi:G\to \operatorname{Aut}(M)$ be some homomorphism. \begin{enumerate}
			\item We say that $\chi$ satisfies property $P$ if the induced map $H^1(G,M)\to H^1(G,M/p)$ is onto.
			\item We say that $M$ satisfies property $Q$ if there exists a unique homomorphism $\chi:G\to \operatorname{Aut}(M)$ which satisfies property $P$.
		\end{enumerate}
	\end{defn} 
	In \cite[Section 5]{labute1966demuvskin} it was proved that:
	\begin{lem}\label{sufficient criterion for s(G)=0}
		Let $G$ be a pro-$p$ Demushkin group. The character $\chi : G \to \operatorname{Aut}(I)$ associated to $G$ has property $P$. Moreover, if there exists a  homomorphism $\chi:G\to \Z_p$ with property $P$, then $s(G)=0$. 
	\end{lem}
	A careful examination of the proof shows that it holds for Demushkin group of arbitrary infinite rank.
	
	The next result \cite[proposition 12]{labute1966demuvskin} applies to any pro-$p$ group $G$.
	\begin{prop}\label{equivalent criterion}
		Let $M$ be a $G$-module which is isomorphic either to $\Q_p/\Z_p$ or to $\Z/q$ for some $p$-power $q$, and let $\chi:G\to \operatorname{Aut}(M)$ be the associated character. Then $\chi$ has property $P$ if and only if any map from a minimal generating set of $G$ to $M$ converging to 1 can be extended to a crossed homomorphism $G\to M$. 
	\end{prop}
	Now we give some results regarding the dualizing module of a Demushkin group of countable rank.
	\begin{lem}\label{property Q}
		Let $G$ be a Demushkin group of arbitrary rank, and $I$ be the dualizing module of $G$, then $\operatorname{Aut}(I)$ has property $Q$.
	\end{lem}
	\begin{proof}
		Let $\chi:G\to \operatorname{Aut}(I)$ be the associated character. Then $\chi$ has property $P$. Expressing $\operatorname{Aut}(I)$ as an inverse limit of finite groups $\Z/q_i^{\times}$, one can express $\chi$ as the inverse limit over some set $I$ of homomorphisms $\chi_i:G_i\to \Z/q_i^{\times}$ were $G_i$ are finitely generated Demushkin groups. We claim that $\chi_i:G_i\to \Z/q_i$ satisfies property $P$. For this we will use the equivalent criterion in Proposition \ref{equivalent criterion}. Indeed, let $x_1,...,x_n$ be a basis of $G_i$. We can lift it to a basis $\{x_i'\}$ of $G$ such that $x_i'\to x_i$ if $i=1,..,n$ and $1$ otherwise. Recall that $\chi$ has property $P$. Now for every choice of elements $D(x_i)$, we can set $D(x_i')=D(x_i)$ if $1\leq i\leq n$ and $0$ otherwise, and get a crossed homomorphism $D:G\to J\to J_i$ that splits through $G_i$. The same proof shows, in fact, that every homomorphism $\sigma:G\to \operatorname{Aut}(I)$ which has property $P$ can be presented as an inverse limit of homomorphisms $\sigma_i:G_i\to \Z/q_i^{\times}$ having property $P$. Thus it is enough to prove that $\Z/q$ has property $Q$ for every finitely generated Demushkin group. We will prove it for $p\ne2$, the proof for $p=2$ is similiar, considering the possible forms of the 1-relator $r$. Recall that for $p\ne2$, a finitely generated pro-$p$ Demushkin group has the form $G/r$ for $r=x_1^q[x_1,x_2]\cdots [x_{2n-1},x_{2n}]$ or $r=[x_1,x_2]\cdots [x_{2n-1},x_{2n}]$. Taking $D_i$ to be the characters defined by $D_i(x_j)=\delta_{ij}$, one concludes that $\chi(x_i)=0$ for all $i\ne 2$ and $\chi(x_2)=(1-q)^{-1}\lor 0$ correspondingly. Hence, we are done.
	\end{proof}
	Now we determine the possible invariants $\Img(\chi)$ for the Demushkin groups $G$ with $q(G)\ne 2$.
	\begin{prop}\label{p=image}
		Let $\chi: G\to \mathrm{Aut}(I)=(\Z_p/s(G)\Z_p)^{\times}$ be the character of $G$ associated to the dualizing module. Then, if $q(G)\ne 2$, $\operatorname{Im}(\chi)=1+q(G)\Z_p/s(G)\Z_p$.
	\end{prop}
	\begin{proof}
		Express $G=\invlim G_i$ as an inverse limit of finitely gnerated Demushkin groups. Since $G/G'=\invlim G_i/G_i'$ we get that if $q(G)\ne 0$, then we can assume that for every $i$, $q(G_i)=q(G)$. Otherwise $q(G_i)$ can be 0 for all $i$, or can be arbitrary large. Express  $\chi$ as an inverse limit $\chi_i:G_i\to \Z/q_i^{\times}$ where $G_i$ are finitely generated Demushkin groups. It is enough to show that for every $i$, $\Img(\chi_i)=1+q(G_i)\Z_p/s(G)  $, then the result follows from standard inverse limit argument. Indeed, as $\chi$ satisfies property $P$, so is $\chi_i$ for every $i$. Hence, as in the proof of Lemma \ref{property Q}, $\chi_i$ must have the form  $x_i=0$ for all $i\ne 2$ and $x_2=(1-q)^{-1}\lor 0$ correspondingly to the value of $s(G)$. Hence, $\Img(\chi_i)=1+q(G_i)\Z_p/s(G)\Z_p $
	\end{proof}
	In the countable case a pro-$p$ Demushkin group is completely determined by $q(G)$ and $s(G)$ whenever $q(G)\ne 2$. This is far from being the case in uncountable rank.
	\begin{prop}\label{group for every form}
		Let $p$ be a prime. For every nondegenerate alternate bilinear form $(V,\varphi)$ over $\F_p$ and $q\ne 2$ equals to power of $p$ or 0, there is a Demushkin group $G$ with $q(G)=q$ whose cup-product bilinear form $H^1(G)\cup H^1(G)\to H^2(G)$ is isomorphic to $(V,\varphi)$, and $s(G)=0$
	\end{prop}
	\begin{proof}
		Let $(V,\varphi)$ be a nondegenerate alternate bilinear form. Let $0\ne v_1\in V$. Since $\varphi(v_1,v_1)=0$, there exists some $v_2\ne v_1$ such that $\varphi(v_1,v_2)\ne 0$, We can assume $\varphi(v_1,v_2)=1$, so by alternating $\{v_1,v_2\}$ is a symplectic basis to $V'=\Span\{v_1,v_2\}$ and in particular $V'$ is nondegenerate. By Lemma \ref{Projections} we can complete $\{x_1,x_2\}$ to a basis of $V$ by taking a basis for the orthogonal complement of $V'$. Now we construct a Demushkin group as follows: Let $F$ be a free pro-$p$ group over $\{x_i\}_{i\in \dim V}$, and define $r=x_1^q[x_1,x_2]\prod_{1\ne i<j} [x_i,x_j]^{\alpha_{ij}}$ or $r=[x_1,x_2]\prod_{1\ne i<j} [x_i,x_j]^{\alpha_{ij}}$ for $\alpha_{ij}=\varphi(v_i,v_j)$ for all $1\ne 1<j$. Then the bilinear form of $(H^1(G),\cup)$ is isomorphic to $(V,\varphi)$ and $q(G)=q$. We will show that $s(G)=0$. By Lemma \ref{sufficient criterion for s(G)=0}, it suffices to build a homomorphism $\sigma:G\to \mathbb{Z}_p^{\times}$ with property $P$. We do it by setting $\theta(x_2)=(1-q)^{-1}$ and $\theta(x_i)=1$ for all $i\ne 2$ in the first case, and the trivial map in the second case. One checking that every crossed homomorphism $D:F\to \mathbb{Z}_p$  vanishes on $r$, and hence property $P$ is satisfied.
	\end{proof}
	Proposition \ref{group for every form} and Theorem \ref{maximal number of forms} together imply:
	\begin{thm}
		Let $p$ be a prime number. For every cardinal $\mu>\aleph_0$ and $q\ne 2$ equals to a power of $p$ or 0, there are $2^{mu}$ pairwise nonisomorphic pro-$p$ Demushkin groups $G$ with $q(G)=q$ and $s(G)=0$. 
	\end{thm}
	\begin{proof}
		The only thing left to observe is that $G_1\cong G_2$ implies isometric of the bilinear forms $(H^1(G_1),\cup)\cong (H^1(G_2),\cup)$ but this is straightforward.
	\end{proof}
	For the case of $p=2$ and $q(G)=2$ we have an analogue result. First we give a list of the subgroups of $\Z_2^{\times}\cong \{\pm1\}\times (1+4 \Z_2)$. Let $f\geq 2$ be some integer or $\infty$ and define $2^{\infty}=0$. Set $U^{(f)}_2=\langle 1+2^f\rangle$ and $U^{[f]}_2=\langle -1+2^f\rangle$. Then every subgroup equals to one of the following options for some $f=2^h, h\in \mathbb{N}\cup \{\infty\}$.
	\begin{enumerate}
		\item $U^{(f)}_2$.
		\item $U^{[f]}_2$.
		\item $\{\pm 1\}\times U^{(f)}_2$.
	\end{enumerate}
	\begin{thm}
		For every  uncountable cardinality $\mu$, $t\in \{-1,0,1\}$, and $f\geq 2$, there are  $2^{\mu}$ pairwise nonisomorphic pro-$2$ Demushkin groups with $q(G)=2,s(G)=0$ and the following pairs of invariants, where $t(G)$ denotes $t(\cup)$ for the cup product bilinear form $H^1(G)\cup H^1(G)\to\F_2$ :
		\begin{itemize}
			\item $t(G)=1$,  $\operatorname{Im}(\chi)=U^{[f]}_2$.
			\item $t(G)=1$,  $\operatorname{Im}(\chi)=\{\pm 1\}\times U^{(f)}_2$.
			\item $t(G)=-1$, $\operatorname{Im}(\chi)=\{\pm 1\}\times U^{(f)}_2$.
		\end{itemize} 
		
	\end{thm}
	\begin{proof}
		For every bilinear form $\varphi$ over $\F_2$ with $t(\varphi)=1$, and every $A\leq\Z_2^{\times}$ of the form $U^{[f]}_2$ or $\{\pm 1\}\times U^{(f)}_2$,  we construct a pro-2 Demushkin group $G$ with $s(G)=0$, whose cup product bilinear form is isomorphic to $\varphi$ and $\Img(\chi)=A$. In addition, for every $\varphi$ over $\F_2$ with $t(\varphi)=1$, and $A\leq\Z_2^{\times}$ of the form $\{\pm 1\}\times U^{(f)}_2$,we construct a pro-2 Demushkin group $G$ with $s(G)=0$, whose cup product bilinear form is isomorphic to $\varphi$ and $\Img(\chi)=A$. Then by Theorem \ref{forms over F_2}, we are done.
		\begin{itemize}
			\item Case 1: Let $\varphi$ be a non alternate bilinear form on $V$ of dimension $\mu$ over $\mathbb{F}_2$ such that $A=\ker(\beta)\ne V$ and $0\ne A^{\perp}\subseteq A$. Let $v_1$ be such that $\varphi(v_1,v_1)=1$ and $v_2\in A^{\perp}$. Then $\varphi(v_1,v_2)\ne 0$ and we can assume $\varphi(v_1,v_2)=1$. In addition, $\varphi(v_2,v_2)=0$. Complete $v_2$ to a basis $\{v_i\}_{i\geq 2}$ for $A$, then $\{v_i\}_{i\geq 1}$ is a basis of $V$. Replace each $v_i, i\geq 3$ by $v_i'=v_i+\varphi(v_i,v_1)v_2$, we get that $\varphi(v_1,v_i)=0$ for all $i\geq 3$.  So we can present the form as an orthogonal sum of nondegenerate subspaces $\operatorname{span}\{v_1,v_2,\}\perp \operatorname{span}\{v_i'\}_{i\geq 3}$.  Define a pro-2 group $G$ with the relation $x_1^{2+2^{f}}[x_1,x_2]\prod_{i,j>2} [x_i,x_j]^{\varphi(v_i',v_j')}$. We can define a map $G\to \mathbb{Z}_2^{\times}$ via $x_2\to -(1+2^{f})^{-1},  x_i\to 1, \forall i\ne 2$. Then the image is $U_2^{[f]}$, and one checks that every crossed homomorphism from the free pro-2 group on $\{x_i\}$ vanishes on $r$, so it has property $P$.
			\item Case 2: We start as in the previous case. Now we may assume that $\varphi(v_3,v_4)=1$. Replace $v_i'$ by their projections $v_i''$ over $\operatorname{span}\{v_3',v_4,\}$. Define a pro-2 group with the relation $x_1^2[x_1,x_2]x_3^{2^f}[x_3,x_4]\prod_{i,j>4} [x_i,x_j]^{\varphi(v_i'',v_j'')}$, $f'\geq f$. We can define a map $G\to \mathbb{Z}_2^{\times}$ via $x_2\to -1, x_4\to (1-2^f)^{-1} x_i\to 1, \forall i\ne 2,4$. Then the image is $\{\pm 1\}\times U_2^{(f)}$, and one checks that every crossed homomorphism from the free pro-2 group on $\{x_i\}$ vanishes on $r$, so it has property $P$.

			\item  Case 3: Let $\varphi$ be a bilinear form on $V$ of dimension $\mu$ over $\mathbb{F}_2$ such that $t(G)=-1$. Then there is an element $v_1\notin \ker(\beta)\land v_1\perp \ker{\beta}$. By Lemma \ref{Projections} $V$ is an orthogonal sum $\operatorname{span}\{x_1\}\perp \ker(\beta)$. So $\ker(\beta)$ is an orthogonal sum of $\operatorname{span}\{v_2,v_3\}\perp V'$. Define $r=x_1^2x_2^{2^f}[x_2,x_3]\prod_{i,j>3} [x_i,x_j]^{\varphi(v_i,v_j)}$, and let $\chi$ be the character defined by $\chi(x_1)=-1,\chi(x_3)=(1-2^f)^{-1}, \chi(x_i)=1\forall i\ne 1,3$. Then the image is $\{\pm 1\}\times U_2^{(f)}$, and one checks that every crossed homomorphism from the free pro-2 group on $\{x_i\}$ vanishes on $r$, so it has property $P$.
			\item $t=0:$ Let $\varphi$ be a bilinear form on $V$ of dimension $\mu$ over $\mathbb{F}_2$ such that $t(G)=0$. Build a Demushkin pro-2 group by the relation $r=\prod[x_i,x_j]^{\varphi(v_i,v_j)}$ for some basis $\{v_i\}$ of $V$. 
		\end{itemize}
		
	\end{proof}
	For finitely generated Demushkin groups we have the following equivalence, which in fact holds in the more general context of Poincare-Duality groups:
	\begin{thm}\cite[Theorem 3.7.2]{neukirch2013cohomology}
		Let $G$ be a finitely generated pro-$p$ group. The following are equivalent:
		\begin{enumerate}
			\item $G$ is a Demushkin group.
			\item $\operatorname{cd}(G)=2$ and $I\cong \Q_p/\Z_p$, where $I$ stands for the dualizing module.
			\item $\operatorname{cd}(G)=2$ and $_pI\cong \F_p$, where $_pI$ denotes the additive subgroup of $I$ of elements of order dividing $p$.
		\end{enumerate}
	\end{thm}
	In the infinite dimension case there are examples of Demushkin groups whose dualizing module is different then $ \Q_p/\Z_p$. We will give such examples for every rank and every value on $q(G)$ in Proposition \ref{demushkin groups with s not 0}. However, we can still prove the following equivalence:
	\begin{thm}
		Let $G$ be a pro-$p$ group. The following are equivalent:
		\begin{enumerate}
			\item $G$ is a Demushkin group.
			\item $\operatorname{cd}(G)=2$ and $_pI\cong \F_p$, where $I$ is the dualising module of $G$ and $_pI$ denotes the additive subgroup of $I$ consisting of elements of order dividing $p$.
		\end{enumerate}
	\end{thm}
	\begin{proof}
		That 1 implies 2 was already observed in \cite[Section 1.3]{labute1966demuvskin} for Demushkin groups of countable rank. However, since by Corollary \ref{open subgroups} every open subgroup of a Demushkin group of arbitrary rank is Demushkin, the same argument works in general.
		
		We now prove that $2$ implies $1$. Recall that by the duality property of $I$, for every discrete $p$-torsion $G$- module $A$, $H^2(G,A)^*\cong \operatorname{Hom}_G(A,I)$. In particular, if $A$ is annihilated by $p$, then $H^2(G,A)^*\cong \operatorname{Hom}_G(A,{_pI})$. By assumption, we get $H^2(G,A)^*\cong \operatorname{Hom}_G(A,\F_p)$. Again, for $G$- modules annihilated by $p$, $\operatorname{Hom}_G(A,\F_p)=\operatorname{Hom}_G(A,\Q_p/\Z_p)$, hence $H^2(G,A)^*\cong \operatorname{Hom}_G(A,\Q_p/\Z_p)\cong (A^*)^G\cong H^0(G,A^*)$. As a result, $H^2(G,\F_p)\cong H^0(G,\F_p^*)^*\cong \F_p$. By definition of the dualizing module, the isomorphism is induced by the cup product pairing $H^2(G,A)\cup H^0(G,A^*)\to H^2(G,\Q_p/\Z_p)\xrightarrow{\operatorname{tr}} \Q_p/\Z_p$. Notice that for finite modules we can reverse the isomorphism by taking $A^*$ instead of $A$. I.e, the cup product induces isomorphism $H^0(G,A)\to H^2(G,A^*)^*$. Since every discrete $G$-module is a direct limit of finite modules, we get this isomorphism holds also for infinite discrete $G$-modules annihilated by $p$.
		
		We left to show that the cup product bilinear form $H^1(G)\cup H^1(G)\to \F_p$ is nondegenerate. This is equivalent to saying that the homomorphism induced by the cup product $H^1(\F_p)\to \operatorname{Hom}(H^1(G),\F_p)\cong H^1(G)^*$ is injective. The last isomorphism follows form the fact that $H^1(G)$ is annihilated by $p$. Notice that since $\F_p\cong \F_p^*$ as $G$-modulus, this is equivalent to the injectivity of the map induced by the cup product $H^1(G,\F_p)\to H^1(G,\F_p^*)^*$. Now look at the following exact sequence: $0\to \F_p\to \operatorname{Ind}^G(\F_p)\to A\to 0$. All the $G$-modules are discrete and annihilated by $p$. Taking its dual, we get the following commutative diagram, when the vertical arrows induced by the cup product.
		$$
		\xymatrix@R=14pt{ H^0(G,\operatorname{Ind}^G(\F_p)) \ar[d] \ar[r] & H^0(G,A) \ar[d] \ar[r] &  H^1(G,\F_p) \ar[d]  \ar[r]  & 0 \ar[d]\\
			H^2(G,\operatorname{Ind}^G(\F_p)^*)^* \ar[r] & H^{2}(G,A^*)^*  \ar[r] &  H^{1}(G,\F_p^*)^* \ar[r]& H^1(G,\operatorname{Ind}^G(\F_p)^*)^* \\}
		$$
		the isomorphisms of the first two vertical maps yields the injectivity of the third via diagram chasing.
	\end{proof}
	We finish this section with the following characterization of Demushkin groups, which generalizes the result in \cite{dummit1983demuvskin} for finitely generated Demushkin groups.
	\begin{thm}
		Let $G$ be a one related pro-$p$ group. $G$ is a Demushkin group if and only if every open subgroup $U\leq G$ of index $p$ is one related.
	\end{thm}
	\begin{proof}
		By Corollary \ref{open subgroups} every open subgroup of Demushkin group is again a Demushkin group, and in particular one related. Hence, we only need to prove the second direction.
		
		Since $G$ is one related, $G=F/\langle r^F\rangle $ for some free pro-$p$ group and $r\in \Phi(F)$. Assume that $G$ is not Demushkin. Then there exists a nontrivial element in the radical of the cup product bilinear form. Call it $\chi_1 \in H^1(G)$. 
		
		Case 1: Suppose that $G/[G,G]$ is torsion free. Complete $\chi_1$ to a basis $(\chi_i)$ of $H^1(G)$ and take a basis $(x_i)_i$ of $F$ which maps onto the dual basis $(\chi_i^*)_i$ of $G/\Phi(G)$, we notice that $r\equiv r'\mod [[F,F],F]$, where $r'$ is product of commutators which doesn't involve $x_1$.

		Case 2: Suppose that $G/[G,G]$ has torsion subgroup $\Z/q\Z$ where $q=p^n$ for some $n \in \mathbb N$. There exists a unique element $s\in F$ such that $s^{p^n}\equiv r\mod [F,F]$. In that case, take a basis for $F$ which contains $s$. Put $x_1=s$,
		
		Subcase 2(a): Suppose  $\chi_1(x_1)\ne 0$. We may assume that $\chi_1(x_1)=1$. By choosing a basis of the annihilator of $x_1$ in $H_1(G)$ and adding $\chi_1$ to it, we get a basis of $H^1(G)$ whose dual contains $x_1$. Since $x_1^*=\chi_1\in \operatorname{rad}(H^1)(G)$, we get that $r\equiv x_1^qr'\mod [F.F,F]$ where $r'$ is product of commutators which doesn't involve $x_1$. 
		
		Subcase 2(b): Suppose that $\chi_1(x)=0$. Then complete $\chi_1$ to a basis of $\textrm{Ann}_{H^1(G)}(x_1)$, and complete it to a basis $B$ of $H^1(G)$ by adding some $\chi'\in H^1(G)$ such that $\chi'(x_1)=1$. We can choose $x_2 \in F$ such that the images of $x_1$ and $x_2$ are dual to $\chi'$ and $\chi_1$ respectively in the dual basis $B^*$ of $G/\Phi(G)$. We get that $r\equiv x_1^qr'\mod [[F,F],F]$ where $r'$ is product of commutators which doesn't involve $x_2$. 
		
		Set $y=x_1$ in Case 1 and Case 2(a)  and set $y=x_2$ in case 2(b). For every finite subset of the chosen basis of $G$, let $F_J$ be the free pro-$p$ group over $J$ and $r_J$ is the image of $r$ under the map $F\to F_J$ which sends $x_i\to x_i$ for all $i\in J$ and $x_i\to 1$ otherwise.  Then $G=\invlim G_J$ where $J$ runs over the finite subsets of the chosen basis which contain $y$ and $G_J=F_J/\langle r_J^{F_J} \rangle $.  Notice that $y$ doesn't appear in the commutators in $r_J$ and hence the image of $\chi_1$ in $H^1(G_J)$ belongs to the radical of the cup product. Hence, $G_J$ is a finitely generated one-relator pro-$p$ group which is not Demushkin and we have set up the notation for $r_J$ to apply the argument in \cite{dummit1983demuvskin}.
		
		Let $N_J$ be the open subgroup generated by $y^p$ and $x_i$ for all $y\ne x_i\in J$, and take $U_J=N_J/R_J$ where $R_J$ is normal subgroup in $F_J$ generated by $r_J$. Correspondingly we have $N\leq F$ and $U=N/R$. Notice that $R_J$ is generated as a normal subgroup of $N_J$ by $r_{j,i}=r_J^{y^i}$ for all $0\leq i\leq p-1$, and the same holds for $R$ and $N$. By \cite{dummit1983demuvskin} $\dim R_J/R_J^p[N_J,R_J]=p$  and $r_{J,i}$ form a basis. Hence the same holds for $R/R^P[R,N]$. We have 2 options: if $r\in \Phi(N)$ then $N/R$ is a minimal presentation of $U$ and we get $H^2(U)=p$. Otherwise, by \cite{dummit1983demuvskin}, for every $J$ there is $1\leq s_J\leq p-2$ such that $r_{J,1},...,r_{J,s_J}$ are linearly independent modulo $\Phi(N_J)$ and the rest of the elements lie in $\Phi(N_J)$. Let $s=\max s_J$. If $s_J=s$ for some $J$, then it holds for every $J\subset J'$, so we can assume $s_J=s$ for all $J$. Hence, $r_1,...,r_s$ are linearly independent modulo $\Phi(N)$, and the rest belongs to $\Phi(N)$. Thus, $\dim H^1(N)-\dim H^1(U)=s$. By the exact sequence $0\to R\to N\to U\to 0$ we get  
		\[ 0\to H^1(U)\to H^1(N)\to H^1(R)^N\to H^2(U)\to 0.\] We already know that $\dim H^1(R)^N=p$. Hence \[\dim H^2(U)=\dim H^1(R)^N- \dim H^1(N)+\dim H^1(U)=p-s>1.\]
		
	\end{proof}
	\section*{Demushkin groups of uncountable rank as absolute Galois groups}
	Let $F$ be a field. Denote by $G_F$ its absolute Galois group, and by $G_F(p)$ its maximal pro-$p$ quotient. Notice that $G_F(p)$ is isomorphic to $\operatorname{Gal}(F(p)/F)$ where $F(p)$ denotes the maximal $p$-extension of $F$- i.e, the compositum of all its finite $p$-extension, and hence it is sometimes called "The maximal pro-$p$ Galois group of $F$ (see, for example, \cite{blumer2023groups}).

	Fix a prime $p$ and assume from now that $F$ contains a primitive $p$-th root of unity $\rho$. Notice that this implies
	that $\operatorname{char}(F)\ne p$. This assumption has no restrictions on our results since by \cite[\S 2. Proposition 3]{serre1979galois} for fields $F$ of  characteristic $p$, $G_F(p)$ is a free pro-$p$ group, and in particular, not Demushkin.

	For $G_F$ we have an arithmetic interpretation of the first cohomology groups. More precisely: $H^1(G_F)\cong F^{\times}/{F^{\times}}^p$, and $H^2(G_F)\cong _p\Br(F)$,  where $\Br(F)$ stands for the Brauer group of $F$, and $ _p\Br(F)$ denotes the subgroup of elements of order dividing $p$. We also have an interpretation of the cup product which is due to Serre: By the Kummer isomorphism $H^1(G_F) \cong F^{\times}/{F^\times}^p$ every element in $H^1(G_F)$ is represented by an element of $F^{\times}$. We denote the elements of $H^1(G)$ by $(a)$ where $a\in F^{\times}$. Now we have the following formula: $(a)\cup(b)=\left (\dfrac{a,b}{F,\rho}\right )$ where $\left (\dfrac{a,b}{F,\rho} \right )$ denotes cyclic algebra generated by two elements $s,t$ over $F$ subject to the relations $s^p=a,t^p=b$ and $sr=\rho ts$.
	
	One of the central problems of number theory is identifying which profinite groups can occur as absolute Galois groups. Since pro-$p$ groups are easier to deal with, a simpler version of this question is: Which pro-$p$ groups can occur as maximal pro-$p$ Galois groups of fields? Recall that by the Artin-Schreier theorem the only nontrivial finite group which can occur as an absolute Galois group is $C_2$- which is the only finite Demushkin group. A few restrictions on the possible properties of absolute, and maximal pro-$p$, Galois groups are already known:
	
	Let $F$ be a field containing a primitive root of unity, and let $G_F$ be its absolute Galois group. Then $G_F$ has a natural action on $\mu_{p^{\infty}}=\bigcup_{n\in \mathbb{N}} \mu_{p^n}\subseteq \bar{F}$, the subset of all $p^n$-roots of unity. Since $\mu_{p^{\infty}}\cong \Q_p/\Z_p$, this action induces a homomorphism $f: G_F\to \operatorname{Aut}(\Q_p/\Z_p)\cong \mathbb{Z}_p^{\times}$. Since $ \rho \in F$ the image of $f$ is a pro-$p$ group and $f$ induces a homomorphism $G_F(p) \to \mathbb{Z}_p^{\times}$. In \cite[Theorem 2.2]{minavc1992pro} it was shown that this homomorphism has property $P$. The same conclusion was proved in  \cite{minavc1991demuvskin} also in case $F$ doesn't contain a primitive $p$-th root of unity. Hence, applying Proposition \ref{sufficient criterion for s(G)=0}, we conclude:
	\begin{prop}
		Let $G$ be a pro-$p$ Demushkin group which is isomorphic to the maximal pro-$p$ Galois group of some field $F$. Then $s(G)=0$.
	\end{prop}
	Obviously, this is not always the case, as is shown in the following proposition:
	\begin{prop}\label{demushkin groups with s not 0}
		For every prime $p$, $q$ a power of $p$ of $0$, $q'>q$ a power of $p$ and $\mu<\aleph_0$ a cardinal, there exists a pro-$p$ Demushkin group of rank $\mu$ such that $q(G)=q$ and $s(G)=q'$.
	\end{prop}
	\begin{proof}
		Let $G\cong F/\langle r^G \rangle$ where $F$ is the free pro-$p$ group generated by $\{x_i\}_{i<\mu}$ and $r=x_0^q[x_0,x_1]\prod x^{q'}_{\lambda+2n}[x_{\lambda+2n},x_{\lambda+2n+1}]$ where $\lambda$ runs over the set of all limit ordinals $\lambda<\mu$. One immediately sees that $G$ is a pro-2 Demushkin group of rank $\mu$ and $q(G)=q$. We will prove that $s(G)=q'$. This is done in the same way as the proof of \cite[Theorem 4]{labute1966demuvskin}. First assume that there exists a homomorphism $\sigma:G\to (\Z_p/q'')^{\times}$ satisfying property $P$ for some $q''$ power of $p$ or 0. Applying $D_i=\delta_{ij}$ for all $i<\mu$ and setting $r$ to 0, we deduce that $\sigma(x_1)=(1-q)^{-1}$, $\sigma(x_{\lambda+2n})=1$ and $\sigma(x_{\lambda+2n+1})=(1-q')^{-1}$. for all $\lambda<\mu$ limit, and all $n\in \omega$. Since $\sigma$ is continuous, $\sigma|_\{x_i\}$ mush be convergent to 1, i.e, only finite number of elements are allowed to be outside a given open subgroup. That can be achieved only when $q''\leq q'$. In particular, since the character induced by the dualizing module has property $P$, we conclude that $0\ne s(G)\leq q'$. On the other hand, we can construct a homomorphism $\sigma: G\to (\Z/q')^{\times}$ by letting $\sigma(x_1)=(1-q)^{-1}$ and $\sigma(x_i)=1$ if $i \not =1$. One checks that every crossed homomorphism $D:F\to (\Z/q')^{\times}$ vanishes on $r$, and hence $\sigma$ has property $P$. By \cite{labute1966demuvskin}. proof of Theorem 4, this implies $q'\leq s(G)$ and we are done.
	\end{proof}
	In \cite{minavc1991demuvskin} it was proved that for every prime $p$ and $q\ne 2$ a power of $p$ or 0, a Demushkin group $G$ of rank $\aleph_0$ with $q(G)$ can be realized as a maximal pro-$p$ Galois group of some field if and only if $s(G)=0$. Unfortunately, the proof relies on the classification of pro-$p$ Demushkin groups of countable rank by these invariants, which is far from being the case for uncountable rank. However, we can point to a strong connection between general Demushkin groups and absolute Galois groups, as shown in the following theorem. The case $q=2$ is more delicate and we will deal with it later,
	\begin{thm}\label{demushkin as galois group}
		Let $p$ be a prime number, $q\ne 2$ be a power of $p$ or $0$, and $\mu>\aleph_0$ a cardinal. There exist a field $F$ whose absolute Galois group is a Demushkin group $G$ of rank $\mu$ and $q(G)=q$.
	\end{thm}
	
	The proof will be done in several steps. First we need the following useful lemma, which is an immediate consequence of \cite[Proposition 19.6]{pierce1982associative}:
	\begin{lem}\label{useful lemma}
		Let $A$ be non trivial algebra in $\operatorname{Br}(K)$, and let $K(x)$ be a transcendental extension of $K$. Then the scalar extension algebra $A_{K(x)}:=A\otimes_KK(x)$ represents a non-trivial element of $\mathrm{Br}(K(x))$. 
	\end{lem}
	\begin{prop}\label{building the field}
		For every prime $p$ and $q\ne 2$ a power of $p$ of 0, there exists a field $K$ which satisfies:
		\begin{enumerate}
			\item $|K^{\times}/{K^p}^{\times}|=\mu$.
			\item $_p\operatorname{Br}(K)\ne 0$.
			\item $K\cap \mu_{p^{\infty}}=\mu_q$ if $q\ne 0$ or $\mu_{p^{\infty}}\subseteq K$ if $q=0$.
		\end{enumerate}
	\end{prop}
	\begin{proof}

		Let $K'=\mathbb{Q}(\mu_q)$ for $q\ne 0,2$ or $ K'=\mathbb{Q}(\mu_p^{\infty})(x,y)$ for $q=0$. By \cite{minavc1991demuvskin}  $H^2(G_{K'})\cong {_p(\operatorname{Br}(K'))}\ne 0$. We choose some $0\ne A\in {_p(\operatorname{Br}(K'))}$. Now let $K=K'[x_i]_{i\in \mu}$ be the field constructed from $K'$ by recursion as follows: For every $\lambda<\mu$ we define:
		\begin{itemize}
			\item If $\lambda=\gamma+1$ then $K'_{\lambda}=K'_{\gamma}[x_{\lambda}]$ for $x_{\lambda}$ a transendental element over $K'_{\gamma}$.
			\item If $\lambda$ is a limit ordinal, then $K'_{\lambda}=\bigcup_{\gamma<\lambda} K'_{\gamma}[x_{\lambda}]$ for $x_{\lambda}$ a transcendental element over $\bigcup_{\gamma<\lambda} K'_{\gamma}$.
		\end{itemize}
		Let $K=\bigcup_{\lambda<\mu} K'_{\lambda}$. We claim that $K$ satisfies the required properties:
		\begin{enumerate}
			\item On one hand, $|K^{\times}/{K^{\times}}^p|\leq |K|=\mu$. On the other hand, we claim that $\{x_i{K^{\times}}^p\}_{i\in \mu}$ is an independent subset. Indeed, assume there is a finite subset $J=\{i_1,...,i_n\}\subseteq\mu$ such that $x_{i_1}^{\alpha_1}\cdots x_{i_n}^{\alpha_{n}}\in {K^{\times}}^p$ with $\alpha_i \in \mathbb N \backslash p\mathbb N$. It means that there is a finite susbset $J'=\{x_{j_1},...,x_{j_m}\}$ such that $x_{i_1}^{\alpha_1}\cdots x_{i_n}^{\alpha_{i_n}}=\left(\dfrac{\sum k_vx_{j_1}^{v_1}\cdots x_{j_m}^{v_m}}{\sum k_ux_{j_1}^{u_1}\cdots x_{j_m}^{u_m}}\right)^p$ where $v_i,u_i\in \mathbb{N}\cup\{0\}$, $k_u, k_v \in K'$ and the sums are finite. Taking $t=\max(J\cup J')$ we get a contradiction to the fact that $x_t$ is transcendental over $\bigcup_{\lambda<t} K'_{\lambda}$.
			\item Let $0\ne A\in _p{\operatorname{Br}(K')}$. We prove that $A_K=A \otimes_{K'}K \ne0$ by transfinite induction. Let $\lambda<\mu$. Assume that for every $\gamma<\lambda$ $A_{K'_{\gamma}}\ne 0$. If $\lambda$ is successor, the we get the result by Lemma \ref{useful lemma}. Otherwise, we claim that $A_{\bigcup_{\gamma<\lambda} K'_{\gamma} }\ne 0$. Indeed, assuming $A_{\bigcup_{\gamma<\lambda} K'_{\gamma}}=0$ it admits $n^2$ matrix units for $n=\deg A$. Every matrix unit is an expression of the form $\sum_{i=1}^{n^2} k_ia_i$ for $a_i\in A$, and hence belong to some $A_{K'_{\gamma}}$ for $\gamma<\lambda$. Hence there is some $\gamma<\lambda$ such that $A_{K'_{\gamma}}$ contains a set of $n^2$ matrix units, and thus it is trivial, a contradiction. Applying Lemma \ref{useful lemma} we get $A_{K'_{\lambda}}\ne 0$. Applying the same argument to $K=\bigcup_{\lambda<\mu} K_{\lambda}$, the result follows.
			\item Obviously, if $\rho_q\in K'$ then $\rho_q\in K$. Let $q$ be a prime of $p$ such that $\rho_q\notin K'$ and assume by contradiction by $\rho_q\in K$. Then there is some finite set $J=\{x_{j_1},...,x_{j_m}\}\subseteq \mu$ such that $\left(\dfrac{\sum k_vx_{j_1}^{v_1}\cdots x_{j_m}^{v_m}}{\sum k_ux_{j_1}^{u_1}\cdots x_{j_m}^{u_m}}\right)^q=1$, all the scalars are different then 0. Since $\rho_q\notin K'$, there exists some $i$ such that one of the $v_i$'s or $u_i$'s is different then 0. Take the maximal such $i$, we get a contradiction to the transcendental of $x_{j_i}$.
		\end{enumerate}
		
	\end{proof}
	Now	we follow the proof of \cite[Main Theorem]{minavc1991demuvskin} with the necessary adjustments to the uncountable rank. 
	\begin{lem}\cite[Lemma 2.1]{minavc1991demuvskin}
		Let $a\in K^{\times}\setminus {K^{\times}}^p$, $A\in {p\operatorname{Br}(K)}$ and $x$ a transcendental element
		over $K$. Then $A_{K(x)}$ and $\left (\dfrac{a,x}{K(x),\rho} \right)$ generate distinct nontrivial subgroups in  $\operatorname{Br}(K(x))$. 
	\end{lem}
	
	\begin{lem}\label{step 2}
		Let $A\in {p\operatorname{Br}(K)}$ and $B$ be an algebra in $\operatorname{Br}(K)$ whose order is a multiple of $p$. If $A\notin \langle B\rangle$ then there is an extension $L$ of $K$ of the same cardinality as $K$ such that $A_L=B_L\ne 1$ in $\operatorname{Br}(L)$. 
	\end{lem}
	\begin{proof}
		This is the same proof of \cite[Lemma 2.2]{minavc1991demuvskin}, noticing that a generic splitting field over $K$ has the same cadinality as $K$.
	\end{proof}
	
	\begin{lem}\cite[Lemma 2.3]{minavc1991demuvskin}.\label{step 3}
		Let $a\in K^{\times}/{K^{\times}}^p$, $A\in {p\operatorname{Br}(K)}$ and $K_1$ be the field obtained from $K$ by first adjoining a transcendental $x$ and then forming the generic splitting field of $A_{K(x)}\otimes \left( \dfrac{a^{-1},x}{K(x),\rho}\right)$. If $b\in K^{\times}\setminus {K^{\times}}^p$ then $b\notin K_1^p$.
	\end{lem}
	\begin{prop}\label{main step}
		Let $K$ be a field of cardinality $\mu$ and $0\ne A\in {p\operatorname{Br}(K)}$.
		Then there is a  field extension $L$ of $K$ of the same cardinality, such that
		\begin{enumerate}
			\item For each $a\in K^{\times} \setminus {K^{\times}}^p$ there exists $h(a)\in L$ with $\left(\dfrac{a,h(a)}{L,\rho}\right)=A_L\ne 1$ and 
			\item if $B\in \operatorname{Br}(K)$ has order divisible by $p$ then either $A_L\in \langle B_L\rangle$ or the order of $B_L$
			is coprime to p. 
		\end{enumerate}
	\end{prop}
	\begin{proof}
		Let $S=\{a_{\alpha}\}_{\alpha<\delta}$ be a set of representatives of $K^{\times}/ {K^{\times}}^p$ order by an ordinal $\delta$ for some ordinal $\delta\leq \mu$. We construct a chain of fields $K_{\alpha}$ for all $\alpha<\delta$ by the following: Assuming we built $K_{\alpha}$ for all $\alpha<\beta$ such that for all $\alpha<\beta$ there exists $h(a_{\alpha})$ such that $\left(\dfrac{a_{\alpha},h(a_{\alpha})}{K_{\gamma},\rho}\right)=A_{K_{\gamma}}\ne 1$ for all $\alpha\leq\gamma<\beta$. If $\beta=\epsilon+1$, build $K_{\beta}$ over $K_{\epsilon}$ as in Lemma \ref{step 3}. If $\beta$ is limit, take $M=\bigcup_{\alpha<\beta}K_{\alpha}$. Notice that, as was proved in the proof of Proposition \ref{building the field}, if $A_{K_\alpha}\ne 1$ for all $\alpha$ then $A_M\ne 1$, and if $b\notin {K_{\alpha}^{\times}}^p$ for all $\alpha$, then $b\notin {M^{\times}}^p$, so we can define $K_\beta$ over $M$ for $a_{\beta}$ as in Lemma \ref{step 3}. 
		
		Now define $M= \bigcup_{\alpha<\delta}K_{\alpha}$. Every $n$-dimensional algebra over $K$ is determined by $n^3$ structure constants, hence we can enumerate the central simple algebras over $K$ by some cardinal $\delta$, $\delta\leq \mu$. We build a chain of field extensions $\{M_{\alpha}\}_{\alpha<\delta}$ as follows: Assume we built $M_{\alpha}$ for all $\alpha<\beta$ such that for all $K$-csa $B_{\alpha}$, $(B_{\alpha})_{M_\gamma}\ne 1$ for all $\alpha\leq \gamma<\beta$, and if the order of $(B_{\alpha})_{\delta}$ is divisible by $p$ for some $\delta<\alpha$ then $A_{M_{\alpha}}\in \langle (B_{\alpha})_{M_\alpha}\rangle$. Then if $\beta=\alpha+1$ we build $M_{\beta}$ over $M_{\alpha}$ for  $B_{\beta}$ as in Lemma \ref{step 2}. If $\beta$ is limit we set $L=\bigcup_{\alpha<\beta}M_{\alpha}$.  Notice that all the required properties are satisfied by $L$. Then we build $M_{\beta}$ over $L$ for $B_{\beta}$ as in Lemma \ref{step 2}.  Eventually take $L=\bigcup_{\alpha<\delta} M_{\alpha}$. Notice that $|L|=|K|$. 
	\end{proof}
	The last step we need is:
	\begin{prop}\label{last step}
		Let $K$ be a field such that $_p\operatorname{Br}(K)\ne 0$ and fix a nontrivial algebra $A\in _p\operatorname{Br}(K)$. Then $K$ admits a field extension $F$ of the same cardinality, such that:
		\begin{enumerate}
			\item $A_F\ne 1$ in $\operatorname{Br}(F)$.
			\item $_p\operatorname{Br}(F)=\langle A_F\rangle$.
			\item For all $a\in F^{\times}\setminus {F^{\times}}^p$ there exits an element $b\in F^{\times}$ such that $\left(\dfrac{a,b}{F,\rho}\right)\ne 1$.
			\item $F^p\cap K=K^p$.
			\item $F\cap \mu_{p^{\infty}}=K\cap \mu_{p^{\infty}}$.
		\end{enumerate}
	\end{prop}
	\begin{proof}
		This is the same construction presented in \cite[Theorem 2.5]{minavc1991demuvskin}. Let us define a series of field extension $K=L_0\subseteq L_1\subseteq L_2\subseteq ...$ such that $A_{L_i}\ne 0$, as follows: Assume we already defined $L_i$. Let $L_i'$ be the fixed field of some $p$-Sylow subgroup $S_p$ of $G_{L_i}$. By \cite[I-11]{serre1979galois}, the restriction map $H^k(G)\to H^k(S)$ from the $k$'th cohomology group of a profinite group to its $p$-Sylow subgroup is injective for all $K$. Hence $A_{L'_i}\ne 0$. Now define $L_{i+1}$ as in Proposition \ref{main step}, where $K=L'_i$ and $A=A_{L'_i}$. Let $F=\bigcup_{i\in \omega} L_i$. We claim that $F$ has properties (1)-(5).
		\begin{enumerate}
			\item Since $A_{L_i}\ne 1$ for all $i$, $A_F\ne 1$. Let $B\in _p\operatorname{Br}(F)$.
			\item  By \cite[p.10]{pierce1982associative} there is some $i$ such that $B={B_i}_F$ for $B_i\in {_p\operatorname{Br}(L_i)}$. ${B_i}_{L_{i+1}}$ has order which is a multiple of $p$, so by Proposition \ref{main step} $A_{L_{i+1}}\in \langle {B_i}{L_{i+1}}$. Hence $A_F\in \langle B_F=B\rangle$. But as both algebras have order $p$, $B\in {_p\operatorname{Br}}(F)$. 
			\item Let $a\in F$ There exists some $i$ such that $a\in L_i'$. By Proposition \ref{main step} there exists some $h(a)\in L_{i+1}$ such that $\left(\dfrac{a,h(a)}{L_{i+1},\rho}\right)\ne 1$. We claim that for every $i+j$, $\left(\dfrac{a,h(a)}{L_{i+j},\rho}\right)\ne 1$. Indeed, if $\left(\dfrac{a,h(a)}{L_{i+j-1},\rho}\right)\ne 1$ then by same observation as in the previous property, $\left(\dfrac{a,h(a)}{L'_{i+j-1},\rho}\right)\ne 1$ and hence by Lemma \ref{useful lemma} $\left(\dfrac{a,h(a)}{L_{i+j},\rho}\right)\ne 1$. Hence, $\left(\dfrac{a,h(a)}{F,\rho}\right)\ne 1$.
			\item Let $a\in K^{\times}\setminus {K^{\times}}^p$. If $a\in {F^{\times}}^p$ then the cohomology class of $a$ in $H^1(G_F)$ is trivial, and hence by the identification of the cup product $(a)\cup(b)=\left(\dfrac{a,b}{F,\rho}\right)=1$, a contradiction to the previous property.
			\item Follows from the previous property by letting $a=\rho_{p^n}$ a primitive root of unity of the maximal order which belongs to $K$.
		\end{enumerate} 
		We left to show that $|F|=|K|$. We show by induction that $|L_i|=|K|$ for all $i$, and then the result follows immediately. It is enough to show that $|L_i|=|L_{i+1}|$. By Proposition \ref{main step} $|L_{i+1}|=|L'_i|$. So we only left to show that $|L_i|=|L'_i|$. But $L'_i$ is defined to be an algebraic extension of $L_i$ and thus has the same cardinality as $L_i$.
		
	\end{proof}
	We are ready to prove the theorem.
	\begin{proof}[proof of Theorem \ref{demushkin as galois group}]
		Let $K$ be the field constructed in Proposition \ref{building the field} and define construct $F$ as in Proposition \ref{last step}. By the same proof as \cite[Main theorem]{minavc1991demuvskin}, $G_K$ is a pro-$p$ group. It also immediate to see that $G_F$ is Demushkin, since the cup product bilinear form is nondegenerate. We need to show that $\operatorname{rank}(G_F)=\dim H^1(G)=\mu$. By the Kummer isomorphism , it is equivalent to show that $|F^{\times}/{F^{\times}}^p|=\mu$. On one hand, $|F^{\times}/{F^{\times}}^p|\leq |F|=|K|=\mu$. On the other hand, by the fourth property $K^{\times}/{K^{\times}}^p\hookrightarrow F^{\times}/{F^{\times}}^p$, and we chose $K$ such that $K^{\times}/{K^{\times}}^p$. The only thing left to show is that $q(G_F)=q$. By the choice of $K$, $\Img(\sigma)=1+q\Z_p$ where $\sigma:G_F\to \Z_p^{\times}$ is the homomorphism induced by the action of $G_F$ on $\mu_{p^{\infty}}$. As this homomorphism satisfies property $P$, and by Lemma \ref{property Q} $\operatorname{Aut}(I)$ has property $Q$, $\sigma$ equals to the character of $G_F$. Hence by Proposition \ref{p=image} $q(G_F)=q$.
	\end{proof}
	
	We are moving to deal with the case $q=2$. In \cite{minavc1992pro} the case of pro-$2$ Demushkin group of uncountable rank has been studied. Most of the results can be apllied for the general case.\begin{thm}
		Let $G$ be a pro-2 Demushkin group. If $G$ is a maximal pro-$2$ Galois group, then $t(G)\ne 0$. In addition, for every cardinal $\mu$, there exists an absolute Galois group over a field of characteristic 0 which is isomorphic to a pro-2 Demushkin group of rank $\mu$ for every pair of invariants:
		\begin{itemize}
			\item $t(G)=1, \operatorname{Im}(\chi)= U_2^{(f)}, 2\leq f<\infty$
			\item $t(G)=1, \operatorname{Im}(\chi)= U_2^{[f]}, 2\leq f<\infty$
			\item $t(G)=1, \operatorname{Im}(\chi)=\{\pm 1\} \times U_2^{(f)}, 2\leq f<\infty$
			\item $t(G)=-1, \operatorname{Im}(\chi)=\{\pm 1\}\times  U_2^{(2)}, 2\leq f<\infty$
		\end{itemize}
		and only for such pairs. 
		If $F$ is a field of characteristic $p$, whose absolute Galois group is a pro-2 Demushkin group $G$, then $t(G)=1$ and exactly the following options of $\operatorname{Im}(\chi)$ are possible:
		\begin{itemize}
			\item  If $p\equiv 1 (mod4)$, say $p = 1 + 2^ac, a > 2, 2\nmid c$, then the only possibilities for $\operatorname{Im}(\chi)$
			are the groups $U_2^{(b)}, a \leq b < \infty$.	 Moreover, each group
			$U_2^{(b)}, a \leq b < \infty$, occurs as $\operatorname{Im}(\chi)$ for some Demushkin group of rank $\mu$ $G_F$ with $\operatorname{char}(F)=p$. 
			\item If $p\equiv -1 (mod4)$, say $p = -1 + 2^ac, a > 2, 2\nmid c$, then the only possibilities for $\operatorname{Im}(\chi)$
			are the groups $U_2^{[a]}$, and  $U_2^(b), a+1 \leq b < \infty$.	 Moreover, each ot these groups can occur as $\operatorname{Im}(\chi)$ for some Demuskin group $G_F$ of rank $\mu$ with $\operatorname{char}(F)=p$
		\end{itemize} 
	\end{thm}
	\begin{proof}
		The fact that no pro-2 Demushkin group which invariants which have not listed above can occur as a maximal pro-$2$ Galois group was proved in \cite{minavc1992pro}. Let $\mu>\aleph_0$ be a cardinal , we will show that every  option from the above list appear for $\mu$. Let $K'$ be the field constructed for every case in \cite{minavc1992pro} and take $K$ to be the field constructed in Proposition \ref{building the field} over $K'$.  We build the field $F$ as in Proposition \ref{main step}. Then $G_F$ is a pro-$2$ Demushkin group of rank $\mu$. The invariant $\Img(\chi)$ is determined by $F\cap \mu_{2^{\infty}}=K\cap_{2^{\infty}}=K'\cap\mu_{2^{\infty}}$. We left with calculation $t(G_F)$. In \cite{minavc1992pro} it was shown that $t(G_F)=1$ if $\left(\dfrac{-1,-1}{F}\right)=1$ and $t(G)=-1$ otherwise. Hence, if $\left(\dfrac{-1,-1}{K'}\right)=1$, obviously $\left(\dfrac{-1,-1}{F}\right)$. And if $\left(\dfrac{-1,-1}{K'}\right)\ne 1$ then we already shown that $\left(\dfrac{-1,-1}{K}\right)\ne 1$. Take $A=\left(\dfrac{-1,-1}{K}\right)$ we get that $\left(\dfrac{-1,-1}{F}\right)\ne 1$, as required.
	\end{proof}
	In contrary to the countable case, in the uncountable case we still left with much mysterious.
	\begin{question}[Open questions]
		Does every pro-$p$ Demushkin group $G$ of uncountable rank for $p\ne 2$ which satisfies $s(G)=0$ occur as a maximal pro-$p$ Galois group of a field?
		
		Does every pro-$2$ group of uncountable rank with $s(G)=0$ and the pairs of invariants listed above  occur as a maximal pro-$2$ Galois group of a field?
	\end{question}
	We end this section by showing that for $p\ne 2$, pro-$p$ Demushkin groups of arbitrary rank satisfy some properties of maximal pro-$p$ Galois groups.
	\begin{prop}
		Let $p\ne 2$. Every pro-$p$ Demushkin group is Bloch-Kato.
	\end{prop}
	\begin{proof}
		A pro-$p$ group $G$ is said to Bloch-Kato, if for every closed subgroup $H\leq G$ $H^{\bullet}(H)$ is a quadretic algebra, meaning that it is generated by elements of the first level, modulo relations of the second level. This property is inspired by the positive solution to the Bloch-Kato Conjecture, states that $H^{\bullet}(G_F(p))\cong K^M_*(F)$ for every field $F$.  By Corollary \ref{open subgroups}, every subgroup of a Demushkin group is either Free or Demushkin. For a free pro-$p$ $F$ group the cohomological dimension is 1 and the claim follows immediately, by letting the relations to be $a\cup b=0$ for all $a,b\in H^1(F)$. We left to show that the cohomology ring of a Demushkin group is quadratic. Let $H$ be a Demushkin group. By Theorem \ref{inverse limit}, Every Demushkin group can be expressed as a projective limit of finitely generated Demushkin groups $G\cong \{G_i,\varphi_{ij}\}$. Since $H^2(G_i)\cong H^2(G)\cong \F_p$ one can choose the groups $G_i$ such that $\operatorname{Inf}:H^2(G_j)\to H^2(G_i)$ are injective for every $j\leq i$. As the maps $\varphi:G_i\to G_j$ are onto, the inflations maps $\operatorname{Inf}:H^1(G_j)\to H^1(G_i)$ are injective for every $j\leq i$. Eventually, since $\operatorname{cd}(G_i)=2$ for every $i$, we get that the maps $H^{\bullet}(G_j)\to H^{\bullet}(G_i)$ induced by the inflations are injective for every $j\leq i$, hence by \cite[Proposition 5.1]{quadrelli2014bloch} we are done.
	\end{proof}
	\begin{prop}
		Let $p\ne 2$. Every pro-$p$ Demushkin group satisfies the $3$-vanishing Massey product property, hereditary.
	\end{prop}
	\begin{proof}
		A pro-$p$ group is said to have the $n$-vanishing Massey product property if every homomorphism as follows $\varphi:G\to \overline{U_n(\F_p)}$ admits a homomorphism $\psi:\varphi:G\to \overline{U_n(\F_p)}$ such that $[\varphi(g)]_{i,i+1}=[\alpha(\psi(g))]_{i,i+1}$ for every $g\in G$ and $i$.
		\[
		\xymatrix@R=14pt{ & & &G \ar@{->>}[dd]^(0.3){\varphi}& \\
			&&&&\\
			1 \ar[r] & \mathbb{F}_p \ar[r] & {\begin{bsmallmatrix}
					1&\rho_{1,2}&\chi_{1,3}&...&\chi_{1,n}\\
					&1&\rho_{2,3}&...&\chi_{2,n}\\
					&&\ddots&\ddots&\vdots \\
					&&&1&\chi_{n-1,n}\\
					&&&&1\\
			\end{bsmallmatrix}} \ar[r]^{\alpha}&{\begin{bsmallmatrix}
					1&\rho_{1,2}&\rho_{1,3}&...&\\
					&1&\rho_{2,3}&...&\rho_{2,n}\\
					&&\ddots&\ddots&\vdots \\
					&&&1&\rho_{n-1,n}\\
					&&&&1\\
			\end{bsmallmatrix}}\ar[r]&1\\
		}
		\]
		
		Since every subgroup of a Demushkin group is either free or Demushkin, it is enough to prove the property holds for every Demushkin group. This follows from a standard inverse limit argument, since by Theorem \ref{inverse limit} every pro-$p$ Demushkin group can be expressed as the inverse limit of finitely generated Demushkin groups, which all satisfy the 3-vanishing Massey product property (see \cite[Theorem 4.3]{Minac2017Triple}). In fact, in \cite[Theorem 4.3]{Minac2017Triple} it was proven that every Demushkin group satisfies the $n$-vanising Massey product property for all $n\geq 3$, so by the same proof we can  prove that Demushkin groups of arbitrary rank satisfy the  $n$-vanising Massey product property. As we deal with properties of maximal pro-$p$ Galois group, we state that it is conjectured that every maximal pro-$p$ Galois group satisfies the   $n$-vanising Massey product property for every $n\geq 3$ (see \cite{Minac2016Triple}). However, it has only been proven for $n=3$ (\cite{Efrat2017Triple},\cite{matzri2014triple}). 
	\end{proof}
	\begin{prop}
		Every pro-$p$ Demushkin group is hereditary of $p$-absolute Galois type.
	\end{prop}
	\begin{proof}
		A pro-$p$ group $G$ is said to be of $p$-absolute Galois type if for every $\alpha\in H^1(G)$, the following sequence is exact:
		\[	
		\begin{tikzcd}
			H^1(\ker(\chi),\mathbb{F}_p) \arrow[r,"\operatorname{Cor}_G"]& H^1(G,\mathbb{F}_p) \arrow[r,"\chi \cup"]& H^2(G,\mathbb{F}_p)  \arrow[r,"\operatorname{Res}_{\ker(\chi)}"]&[1.2em]  H^2(\ker(\chi),\mathbb{F}_p) 
		\end{tikzcd}
		\]
		By \cite{lam2023generalized} it is enough to prove that the sequence is exact at $H^2(G,F)$ for every $\alpha\in H^1(G)$. Again, since every subgroup of a Demushkin group is either free or Demushkin, it is enough to prove the property holds for every Demushkin group. Let $H$ be a Demushkin group and let $\alpha\in H^1(H)$. If $\alpha=0$ then $H=\ker(\alpha)$, and the map: $\alpha\cup (-):H^1(H)\to H^2(H)$ is the zero map, so we are done. Otherwise, since the cup product is nondegenerate, the map $ \alpha\cup (-):H^1(H)\to H^2(H)$ is onto, while obviously $\operatorname{Res}_ {\ker{\alpha}}(\alpha\cup \beta)=0$ for all $\beta\in H^1(G)$.
		
	\end{proof}
	\section*{Profinite completions of Demushkin groups}
	In this section we compute the profinite completion of a Demushkin group of infinite rank, and get a new class of examples of absolute Galois groups having absolute Galois completions. Recall that the profinite completion of an abstract group $G$, denoted by $\hat{G}$, is defined as the inverse limit ${\invlim}_{U\unlhd_fG}G/U$ where $U$ runs over the finite index normal subgroups of $G$, equipped with the natural homomorphism $i:G\to \hat{G}$  and it satisfies the following universal property: every homomorphism $f:G\to H$ into a profinite group can be lifted uniquely to a continuous homomorphism $\hat{f}:\hat{G}\to H$. Let $G$ be a profinite group. Considered as an abstract group, $G$ has a profinite completion with an injection $i: G \to \hat G$. We say that $G$ is strongly complete when $i$ is an isomorphism. For pro-$p$ groups, we have the following equivalence:
	\begin{prop}
		Let $G$ be a pro-$p$ group. The following are equivalent:
		\begin{enumerate}
			\item $G$ is finitely generated.
			\item $G$ is strongly complete.
		\end{enumerate}
	\end{prop}
	In \cite{baron-cohomological} the author presented the question: Can the profinite completion of an absolute Galois group also be realized as an absolute Galois group? The purpose of this section is to give a new class of examples with a positive answer. First we need to discuss free pro-$p$ groups. A well know fact (see, for example \cite[Example 3.3.8 (e) for infinite rank]{ribes2000profinite}) states that every free profinite group can be realized as an absolute Galois group. Since every subgroup of a free profintie group is projective, and for pro-$p$ group projectivity equals freeness, we get that the $p$-Sylow subgroups of free profintie groups are free pro-$p$ groups. It also can be proven the $p$-Sylow subgroup of a free profintie group has the same rank of the whole group. As a result for every cardinal $\mu$, the free pro-$p$ group of rank $\mu$ can be realized as an absolute Galois group. However, we give here a direct prove of this result, inspired by Minac\& Ware proof:
	\begin{prop}\label{free as absolute}
		Let $\mu$ be an infinite cardinal, and $G$ be the free pro-$p$ group of rank $\mu$. Then $G$ occurs as an absolute Galois group of some field.
	\end{prop}
	\begin{proof}
		Recall that a pro-$p$ group $G$ is free if and only if $\operatorname{cd}(G)\leq 1$, which is equivakent to $H^2(G)=0$ (see, for example \cite[Chapter 7]{ribes2000profinite}). Hence, our objective is constructing a field $F$ with $|F^{\times}/{F^{\times}}^p|=\mu$, $_p\operatorname{Br}(F)=0$ and such that $G_F$ is a pro-$p$ group. We present a similar construction of Theorem \ref{demushkin as galois group}. Let $K$ be a field such that $|K|=|K^{\times}/{K^{\times}}^p|=\mu$. An example of such a field can taken from Proposition \ref{building the field} for the required $\mu$. As every algebra over $F$ can be determined by $n^3$ constants, we can order $_p\operatorname{Br}(F)$ over some ordinal $\lambda\leq \mu$. We build a field $L_1$ a s follows: let $\alpha<\lambda$ and assume that for every $\beta<\alpha$ we already defined a field $M_{\beta}$, such that for every $\gamma\leq \beta$, ${A_{\gamma}}_{M_{\gamma}}=1$, and $M_{\beta}^p\cap K=K^p$. If $\alpha=\beta+1$ then let $M_{\alpha}$ be the field constructed from $M_{\beta}$ by forming the generic splitting field of ${A_{\alpha}}_{M_{\beta}}$. Otherwise $M_{\alpha}$ will be the generic splitting field forming over $\bigcup_{\beta<\alpha} M_{\beta}$ for the algebra ${A_{\beta}}_{\bigcup_{\alpha<\beta} M_{\alpha}}$. One easily observes that $|L_1|=|K|$, and using Lemma \ref{step 3} and the fact that this property is preserved by direct limit, that $L_1^p\cap K=K^p$. 
		
		Define a series $L_1\subseteq L_2\subseteq...$ as in Proposition \ref{last step} and take $F=\bigcup L_i$. The Theorem follows.
	\end{proof}
	In fact there is a more general and stronger result which can be found in \cite{fried2006field}, which states that profinite projective groups- which are the profinite groups all whose $p$-Sylow subgroups are free- are precisely the absolute Galois groups of pseudo algebraically closed fields. However, we don't need it here.
	
	Before computing the profinite completion of a Demushkin group, we need one more definition.
	\begin{defn}
		A pro-$p$ group $G$ is called \textit{locally free} if every finitely generated closed subgroup of $G$ is free.
	\end{defn}
	\begin{prop}\label{locally free}
		Let $G$ be a pro-$p$ group. Then $G$ is locally free if and only if $\hat{G}$ is free.
	\end{prop}
	\begin{proof}
		Assume first that $\hat{G}$ is free. Let $H\leq G$ be a finitely generated closed subgroup. Since $H$ is finitely generated. it is strongly complete. Hence, every finite-index subgroup of $H$ is open. However, by the basic properties of profinite groups, every open subgroup of $H$ contains the intersection of $H$ with an open subgroup of $G$. Thus, the inclusion $i:H\to G$ induces a monomorphism $\hat{H}\to \hat{G}$. Recalling again that $H$ is strongly complete, we get that $H$ is isomorphic to a closed subgroup of a free pro-$p$ group, and hence $H$ is free by the results in \cite[Chapter 7]{ribes2000profinite}.
		
		For the other direction we shall use the following characterisation of free pro-$p$ groups from \cite[Theorem 7.7.4]{ribes2000profinite}: A pro-$p$ group $L$ is free if and only if every finite embedding problem, i.e, a pair of continuous epimorphisms $\varphi:L\to A, \alpha: B\to A$ is weakly solvable, meaning that there is a continuous homomorphism $\psi:L\to B$ making the following diagram commutative:
		$$\xymatrix@R=14pt{ & L \ar@{->>}[d]^{\varphi} \ar[ld]_{\psi}& \\
			B\ar [r]_{\alpha} &A \ar[r]&1\\
		}$$ 
		
		Now we assume that $G$ is locally free and we wish to prove that $\hat G$ is free. 
		Let the following diagram
		$$\xymatrix@R=14pt{ &\hat{G} \ar@{->>}[d]^{\varphi} & \\
			B\ar [r] &A \ar[r]&1\\
		}$$
		be an embedding problem for $\hat G$. Look at the induced abstract embedding problem for $G$:
		$$\xymatrix@R=14pt{ &G \ar@{->>}[d]^{\varphi\circ i} & \\
			B\ar [r] &A \ar[r]&1,\\
		}$$
		where $i$ denotes the natural homomorphism $i:G\to \hat{G}$. Since $i(G)$ is dense in $\hat{G}$ and $A$ is finite we have that $\varphi\circ i$ is surjective.
		
		Choose a finite set $X=\{x_1,...,x_n\}$ of preimages of $A$. For every finite subset $Y\subseteq G$ Let $$H_Y=\overline{\langle X\cup Y\rangle}$$ be the closed subgroup of $G$ generated by $X\cup Y$. Since $X\subseteq H_Y$ then $\varphi|_{H_Y}:H_Y\to A$ is an epimorphism. Since $H$ is finitely generated it is strongly complete and thus any homomorphism to a profinite group is continuous. So, $ \varphi|_{H_Y}:H_Y\to A$ is a continuous epimorphism. By assumption $H_Y$ is free pro-$p$ group and therefore there is a homomorphism $\psi :H_Y\to B$ such that $\alpha\circ \psi =\varphi|_{H_Y}$.
		
		Denote by $\mathcal{A}_Y$ the set of all continuous week solutions $\psi: H_Y\to B$ such that $\alpha \circ \psi =\varphi|_{H_Y}$. We have that $\mathcal{A}_Y \not = \emptyset$ and since $H_Y$ is finitely generated it follows that $\mathcal{A}_Y$ is finite.
		
		For any pair of finite subsets $Y\subseteq Z \subset G$ the restriction function $f_{ZY}: \mathcal{A}_Z\to \mathcal{A}_Y$ is defined by $\psi \mapsto \psi|_{H_Y}$ for $\psi \in \mathcal{A}_Z$. Whenever $Y\subseteq Y'\subseteq Y''$ we have $f_{Y''Y}=f_{Y'Y} \circ f_{Y''Y'}$. It follows that $\{\mathcal{A}_Y,f_{ZY}\}_{Y\subseteq_f G}$ is a directed system of nonempty finite sets and thus its inverse limit is nonempty.
		
		An element in the inverse limit is an homomorphism $f:\bigcup_{Y\subseteq_f G} H_Y\to B$ which satisfies $ \alpha\circ f=\varphi$ But $\bigcup_{Y\subseteq_f G} H_{Y} = G$. So we got a weak solution $f: G \rightarrow B$ to the embedding problem $\alpha \circ f=\varphi$. By the universal property of the profinite completion $\hat G$, $f$ induces a continuous homomorphism $\hat{f}:\hat{G}\to B$. The density of $i(G)$ in $\hat G$ implies $\alpha \circ \hat f=\varphi$, i.e. $\hat f$ is the required weak solution for the embedding problem of $\hat G$. Therefore $\hat G$ is a free pro-p group as claimed.
	\end{proof}
	\begin{cor}\label{completion of demushkin}
		Let $G$ be a pro-$p$ Demushkin group of infinite rank. Then $\hat{G}$ is a free pro-$p$ group.
	\end{cor}
	\begin{proof}
		By Proposition \ref{locally free} it is enough to prove that every finitely generated closed subgroup of $G$ is free. Since $G$ is not finitely generated, a finitely generated closed subgroup must have an infinite index. Thus by Corollary \ref{open subgroups} we are done.  
	\end{proof}
	Combining theorem \ref{demushkin as galois group}, Proposition \ref{free as absolute} and Corollary \ref{completion of demushkin} we conclude the following:
	\begin{cor}
		For every infinite cardinal $\mu$ there exists a nonfree pro-$p$ absolute Galois group of rank $\mu$ whose profinite completion is an absolute Galois group as well.
	\end{cor}
	In fact, Corollary \ref{completion of demushkin} also gives a negative result to the open question presented by Andrew Pletch in 1982 in his paper \cite{pletch1982local}: Is every locally free pro-$p$ group free?

\end{document}